\documentclass[oneside,11pt]{article}
\usepackage{arxiv/style}
\usepackage[T1]{fontenc}
\usepackage{amsmath,amssymb,amsthm}
\usepackage{graphicx,booktabs,array}
\usepackage{float}
\makeatletter
\@ifpackageloaded{biblatex}{}{\usepackage{natbib}}
\makeatother
\usepackage{microtype}
\usepackage{needspace}
\usepackage[section]{placeins}
\usepackage{hyperref}
\usepackage[nameinlink]{cleveref}
\newif\ifseparatesupplement
\hypersetup{colorlinks=true,linkcolor=blue,citecolor=blue,urlcolor=blue,
  pdftitle={Multiple testing in multi-stream sequential change detection},
  pdfauthor={Sanjit Dandapanthula and Aaditya Ramdas}}
\theoremstyle{plain}
\newtheorem{theorem}{Theorem}[section]

\newtheorem{proposition}[theorem]{Proposition}
\newtheorem{corollary}[theorem]{Corollary}
\theoremstyle{definition}
\newtheorem{definition}[theorem]{Definition}

\newtheorem*{remark}{Remark}
\newcommand{\E}{\mathbb{E}}
\renewcommand{\P}{\mathbb{P}}
\newcommand{\F}{\mathcal{F}}
\newcommand{\N}{\mathbb{N}}
\newcommand{\R}{\mathbb{R}}
\newcommand{\indicator}{\mathbf{1}}

\newcommand{\fdr}{\mathrm{FDR}}
\newcommand{\fdp}{\mathrm{FDP}}
\newcommand{\fwer}{\mathrm{FWER}}
\newcommand{\pfer}{\mathrm{PFER}}
\newcommand{\arl}{\mathrm{ARL}}
\newcommand{\pfa}{\mathrm{PFA}}
\newcommand{\eop}{\mathrm{EOP}}

\newcommand{\ger}{\mathrm{GER}}

\DeclareMathOperator{\Bin}{Bin}
\crefname{theorem}{theorem}{theorems}
\Crefname{theorem}{Theorem}{Theorems}
\crefname{proposition}{proposition}{propositions}
\Crefname{proposition}{Proposition}{Propositions}
\graphicspath{{shared/figures/}{experiments/results/}}

\pretitle{\begin{center}\Large}
\hypersetup{
  colorlinks,
  linkcolor={blue!50!black},
  citecolor={blue!50!black},
  urlcolor={blue!50!black}
}
\makeatletter
\renewcommand\hyper@natlinkbreak[2]{#1}
\def\fps@figure{tbp}
\def\fps@table{tbp}
\makeatother
\theoremstyle{remark}
\newtheorem*{arxivremark}{Remark}

\renewcommand\arraystretch{.6}
\AddToHook{env/table/begin}{\renewcommand{\arraystretch}{1}}
\makeatletter
\renewcommand*\env@cases[1][0.8]{%
  \let\@ifnextchar\new@ifnextchar
  \left\lbrace
  \def\arraystretch{#1}%
  \array{@{}l@{\quad}l@{}}}
\makeatother

\DeclareFieldFormat{citehyperref}{%
  \DeclareFieldAlias{bibhyperref}{noformat}%
  \unskip\bibhyperref{#1}}
\DeclareFieldFormat{citeparenshyperref}{%
  \DeclareFieldAlias{bibhyperref}{noformat}%
  \bibhyperref{#1}}
\DeclareFieldFormat{textcitehyperref}{%
  \DeclareFieldAlias{bibhyperref}{noformat}%
  \bibhyperref{#1%
    \ifbool{cbx:parens}{\bibcloseparen\global\boolfalse{cbx:parens}}{}}}
\savebibmacro{cite}
\savebibmacro{textcite}
\renewbibmacro*{cite}{%
  \ifbool{cbx:parens}
    {\printtext[citeparenshyperref]{\restorebibmacro{cite}\usebibmacro{cite}}}
    {\printtext[citehyperref]{\restorebibmacro{cite}\usebibmacro{cite}}}}
\renewbibmacro*{textcite}{%
  \ifboolexpr{
    (not test {\iffieldundef{prenote}} and test {\ifnumequal{\value{citecount}}{1}})
    or
    (not test {\iffieldundef{postnote}} and test {\ifnumequal{\value{citecount}}{\value{citetotal}}})
  }{\DeclareFieldAlias{textcitehyperref}{noformat}}{}%
  \printtext[textcitehyperref]{\restorebibmacro{textcite}\usebibmacro{textcite}}}

\title{\textbf{Multiple testing in multi-stream sequential change detection}}
\author{Sanjit Dandapanthula\thanks{Department of Statistics and Machine Learning Department, Carnegie Mellon University}\\[.4em]
\texttt{\href{mailto:sanjitd@cmu.edu}{sanjitd@cmu.edu}}
\and Aaditya Ramdas\thanks{Department of Statistics, Stanford University}\\[.4em]
\texttt{\href{mailto:aramdas@stanford.edu}{aramdas@stanford.edu}}}
\date{\vspace{.5cm}\today}
\begin{document}
\maketitle
\begin{abstract}
Multi-stream sequential change detection involves simultaneously monitoring many data streams and raising an alarm when their distributions change. We study the multiple testing issues that arise naturally in this setting. We first point out a fundamental tension between the criteria used in change detection and multiple testing: any procedure with finite average run length (ARL) must have worst-case false discovery rate (FDR) and family-wise error rate equal to one. Given this conflict, we find a middle ground by introducing the \emph{error over patience} (EOP) class of Type I error metrics, which measure error relative to expected monitoring time. To control the EOP, we propose procedures combining e-detectors, which generalize Shiryaev-Roberts and CUSUM methods, with the recent e-Benjamini-Hochberg and e-Bonferroni methods for multiple testing. Our EOP guarantees hold uniformly over changepoints and stopping times with finite expectation, allowing dependence when each e-detector remains valid using information from all streams. If finite ARL is forfeited, using e-processes instead yields ordinary FDR and per-family error control at stopping times. We also prove power and delay bounds for Gaussian changes, establish first-order minimax optimality, and compare our procedures with existing methods in numerical experiments.

\end{abstract}
\section{Introduction}

In single-stream sequential change detection, we observe a sequence of random variables $(X_t)_{t=1}^\infty$ one at a time and our goal is to stop as soon as we detect a change in the data-generating distribution. The unknown time of the change is called the \emph{changepoint}. A common metric considered in this setting is the average run length (ARL), which is the expected time until a false alarm when there is no change. One typically seeks change detection algorithms that have a small detection delay while guaranteeing a lower bound on the ARL.

The multi-stream setting is concerned with detecting changes in many streams. We do not assume that changes, if any, occur at the same time across streams. An algorithm which is effective in the single-stream setting may not be good when deployed without alteration on each of the multiple streams. This is because we would only obtain a guarantee on the per-stream ARL, but not on metrics that evaluate the reported changes across streams.

For instance, we may want control over the false discovery rate (FDR), which is the expected proportion of detections which are false. False detections may be costly, and one may consequently desire to control the family-wise error rate (FWER), the probability that at least one reported detection is false, or the per-family error rate (PFER), the expected number of false detections. These criteria concern which streams are reported to have changed. We therefore study algorithms that identify changed streams, in addition to detecting that a change has occurred somewhere.

We consider monitoring algorithms which report a set of streams at each time, and allow the user to decide when to stop. A \emph{stopping time} is a rule for stopping that uses only information available at the time of the decision. For example, a user might stop when the first change is declared, when several changes are declared, or when there are insufficient resources to continue monitoring. We seek error guarantees that hold for any such choice, including rules that use observations from all streams.

We first show that finite global-null ARL prevents FDR or FWER control below one at every stopping time. Under the \emph{global null}, meaning that no stream ever changes, a finite ARL implies that a first alarm occurs almost surely. If we stop at that alarm, all the reported changes are false. Thus, the worst-case FDR and FWER over stopping times and choices of changepoints equal one. This gives a fundamental tension between finite ARL and error control at arbitrary stopping times. Controlling the error at each fixed time is a different requirement.

We propose \emph{error over patience} (EOP), a class of Type I error metrics for sequential monitoring. EOP divides an error criterion, such as FDR or PFER, by expected monitoring time:
\begin{align*}
    \eop
    =\sup_{\text{changepoints}}\ \sup_{\text{stopping times }\tau}
       \frac{\text{Type I error at }\tau}{\E[\tau]}.
\end{align*}
Here, the denominator is termed \emph{patience}. We take the supremum over stopping times with finite expectation. An EOP bound requires the error to be small when the expected monitoring time is short, and permits it to grow with expected duration. For example, an EOP bound of $\alpha$ gives an error bound of $\alpha B$ for any stopping rule with expected duration at most $B$. A user with an error budget $q$ and expected duration at most $B$ can therefore take $\alpha=q/B$, while still choosing when to report using the observed data. We show that FDR EOP control also implies the familiar global-null ARL bound $1/\alpha$. Its further requirement is to control errors in the reported set at other stopping times, including after some streams have changed. We explain these implications and discuss other denominators in \Cref{sec:setup}.

To construct procedures with these guarantees, we use e-detectors \citep{shin2022detectors}. An \emph{e-detector} is a nonnegative process whose expected value at a stopping time is at most the expected stopping time under the null hypothesis. E-detectors generalize the classical Shiryaev--Roberts (SR) and cumulative sum (CUSUM) procedures to nonparametric settings: the SR statistic and a likelihood-ratio CUSUM statistic are particular examples. \citet{ramdas2026universality} relaxes this expectation condition to obtain a universal representation of change detection procedures satisfying an ARL constraint. We use the stopped expectation bound to control EOP at data-dependent reporting times.

We construct e-d-BH by combining the streamwise e-detectors with the e-Benjamini--Hochberg (e-BH) procedure of \citet{wang2022false}. We also construct e-d-Bonferroni, which applies a common threshold to the e-detectors. With e-d-BH, the evidence in several streams can lower the threshold for declaring a change, allowing streams with less evidence to be included. We prove that the two procedures control FDR EOP and PFER EOP, respectively, and derive lower bounds on the expected time to false declarations. We do not require independence across streams: each e-detector must instead remain valid when stopping can use observations from all streams. We state this condition precisely in \Cref{sec:edetcd} and give Gaussian and nonparametric examples in \Cref{app:filtration}.

We then analyze detection in a Gaussian mean-change model in \Cref{sec:power}. We prove finite-sample bounds on the number of observations sufficient to detect a prescribed fraction of the changes, and use these bounds to control expected delay. The bounds explain how having more changed streams reduces the evidence required by e-d-BH. For independent streams and a fixed number of streams, we also prove a matching first-order minimax lower bound under the FDR EOP constraint. Both e-d-BH and e-d-Bonferroni attain this bound as the error level tends to zero. This establishes their leading-order efficiency for identifying changed streams, while leaving room for differences at finite error levels.

In \Cref{sec:sims}, we compare the performance of our EOP-controlling procedures with existing frequentist and Bayesian identification methods. We evaluate the errors and true discoveries in each method's first report, and include both fixed changepoints and changepoints generated from the Bayesian prior. We find that e-d-BH reports more true changes and reports sooner than a time-adjusted CUSUM procedure in the fixed-change experiments; Bayesian procedures benefit when the data follow their prior. We also show that the gain over e-d-Bonferroni persists with heavy-tailed observations and dependent counts. Finally, we compare with classical global detectors at a common ARL target. These methods can combine evidence to raise an alarm sooner without deciding which streams to report, and the comparison helps explain the additional cost of identification.

\ifseparatesupplement
The Supplementary Material \citep{dandapanthula2026supplement} contains the proofs and additional results. Numbers prefixed by S refer to that document.
\fi
We give the proofs in \Cref{app:proofs-eop,app:proofs-methods,app:proofs-power}, following the order of the main results. In \Cref{app:further-results}, we discuss other discounts and further consequences of the error bounds.

\subsection{Related work}\label{sec:related}

Classical multi-stream change detection often asks when to raise a global alarm, subject to a lower bound on the no-change ARL. Examples include the scalable CUSUM procedures of \citet{mei2010efficient}, the mixture procedures of \citet{xie2013sequential}, and the asymptotic analysis of \citet{chan2017optimal}. These methods can combine moderate evidence from many streams to detect a change quickly, and their model assumptions permit detailed analyses of detection delay. We study the additional problem of identifying which streams changed. Controlling false discoveries in the reported set can require more evidence than raising a global alarm.

Multiple-testing criteria have also been studied in change detection. \citet{bodenham2017continuous} discuss false detections and changes correctly detected in a single stream. For independent streams, \citet{gandy2013nonrestarting} compute p-values from the null distribution of bounded CUSUM statistics and apply BH to control FDR at each fixed time. In the parallel setting, \citet{chen2020false} study FDR-oriented detection under a Bayesian model, while \citet{lu2022optimal} develop a broader framework for balancing detection errors and utility. \citet{chen2023compound} control the expected proportion of changed streams among those still being monitored, with declarations permanently removing streams. The latter two works derive optimality results under their assumptions on the changepoints and data distributions. We treat changepoints as unknown fixed quantities and prove bounds that hold for every choice of these times; our error criterion scales with expected monitoring time.

The work of \citet{lee2026sequential}, whose preprint appeared after the first version of our paper, is especially close to ours. They use \emph{e-values}, nonnegative statistics with null expectation at most one, to control FDR at each deterministic time in a sensor network whose membership can change. For a fixed set of sensors, their basic SHEREF rule is a special case of our fixed-time procedure in \Cref{sec:fixed}: it applies e-BH to SR statistics divided by the current time. Their framework additionally handles sensor entry and removal, and develops methods to improve power. Our main result uses the unnormalized e-detectors to bound error divided by expected monitoring time at an arbitrary stopping time.

\section{Setup and error over patience}\label{sec:setup}

We introduce the monitoring problem and the error criteria used throughout the paper. The proofs for this section are in \Cref{app:proofs-eop}.

\subsection{Multi-stream sequential change detection}\label{sec:background-cd}

Let $\N=\{1,2,\ldots\}$. Suppose we are monitoring $K\in\N$ streams of data, denoted $(X_t^{(k)})_{t=1}^\infty$ for $k\in[K]=\{1,\ldots,K\}$. At time $t$, we observe $(X_t^{(1)},\ldots,X_t^{(K)})$. Each stream has a changepoint $\xi^{(k)}\in\N\cup\{\infty\}$, where $\xi^{(k)}=\infty$ means that there is no change. Stream $k$ is pre-change at times $t<\xi^{(k)}$ and post-change at times $t\geq\xi^{(k)}$. We call the vector of changepoints $\xi=(\xi^{(1)},\ldots,\xi^{(K)})$ the \emph{configuration}. The \emph{global null} is the configuration $\xi_G=(\infty,\ldots,\infty)$ in which no stream changes.

We allow dependence within and across streams. Let $\P$ denote the joint data-generating distribution. The \emph{global filtration} $\F=(\F_t)_{t\geq0}$ records the information available at each time from all streams. Unless specified otherwise, we take
\begin{align*}
    \F_t=\sigma\bigl(X_s^{(k)}:s\leq t,\ k\in[K]\bigr).
\end{align*}
We write $\P_\infty$ and $\E_\infty$ for probability and expectation under a global-null law. A \emph{stopping time} is a random variable $\tau\in\N\cup\{\infty\}$ such that $\{\tau\leq t\}\in\F_t$ for every $t$. All stopping times below refer to $\F$, so they may depend on observations from all streams. In \Cref{app:filtration}, we explain why a guarantee for stopping times based on one stream need not hold when the other streams are also observed. Our guarantees hold for every joint law satisfying the stated assumptions. When taking a supremum over configurations, we also take the supremum over all such laws for each configuration.

A \emph{change monitoring algorithm} is a collection of processes $\varphi=(\varphi^{(1)},\ldots,\varphi^{(K)})$, where $\varphi_t^{(k)}\in\{0,1\}$ is $\F_t$-measurable. We interpret $\varphi_t^{(k)}=1$ as a declaration that stream $k$ has changed by time $t$. After accumulating more evidence, the algorithm may change its decision: we allow $\varphi_t^{(k)}=1$ and $\varphi_s^{(k)}=0$ for $s>t$. A \emph{change detection algorithm} is a pair $(\varphi,\tau)$, which stops at $\tau$ and outputs $(\varphi_\tau^{(k)})_{k=1}^K$ when $\tau<\infty$.

Throughout the paper, write $x\vee y=\max\{x,y\}$ and $x\wedge y=\min\{x,y\}$. Let $R_t=\{k:\varphi_t^{(k)}=1\}$ be the set of declared streams and $\mathcal H_0(t)=\{k:\xi^{(k)}>t\}$ the set of streams that have not yet changed. We write $N_t=|R_t|$ for the number of declarations and $V_t=|R_t\cap\mathcal H_0(t)|$ for the number that are false. The \emph{false discovery proportion} is $\fdp_t=V_t/(N_t\vee1)$. The error metrics we consider are
\begin{align}
    \fdr(\varphi,\xi,t)&=\E\left[\frac{V_t}{N_t\vee1}\right],
    &\fwer(\varphi,\xi,t)&=\P(V_t\geq1),
    &\pfer(\varphi,\xi,t)&=\E[V_t].\label{eq:errors}
\end{align}
We use the same definitions at a stopping time, replacing $t$ by $\tau$. These metrics evaluate the set currently declared. Retaining all past declarations would produce a different set, to which our guarantees need not apply.

For $1\leq\eta\leq K$, let $\tau_\eta^*(\varphi)$ be the first time at least $\eta$ streams are declared. Its expectation under the global null is the \emph{average run length}:
\begin{align}
    \tau_\eta^*(\varphi)&=\inf\{t\geq1:N_t\geq\eta\},
    &\arl_\eta(\varphi)&=\E_\infty[\tau_\eta^*(\varphi)].\label{eq:arl}
\end{align}
As usual, $\inf\varnothing=\infty$. Under a general configuration, we also consider
\begin{align}
    \tau_\eta^*(\varphi,\xi)=\inf\{t\geq1:V_t\geq\eta\},
    \qquad
    \arl_\eta(\varphi,\xi)=\E[\tau_\eta^*(\varphi,\xi)].\label{eq:falsearl}
\end{align}
Thus, $\arl_\eta(\varphi,\xi)$ is the expected time to at least $\eta$ false declarations. The stopping rule defining it depends on the unknown changepoints and is mainly of theoretical interest. The \emph{probability of false alarm} under the global null is
$\pfa(\varphi)=\P_\infty(\tau_1^*(\varphi)<\infty)$.
ARL describes false alarms; we study detection delay separately under a specified alternative in \Cref{sec:power}.

\subsection{Error over patience}\label{sec:newmetrics}

We first show that any procedure with finite global-null ARL cannot control FDR or FWER below one at every stopping time. Under the global null, every declaration is false. A user who waits for the first declaration will therefore receive a false discovery, however long the procedure takes to make it. The same reasoning applies when the user waits for at least $\eta$ declarations. We state this observation below.
\begin{proposition}\label{prop:impossibility}
If $\arl_\eta(\varphi)<\infty$, then
\begin{align*}
    \fdr(\varphi,\xi_G,\tau_\eta^*(\varphi))
    =\fwer(\varphi,\xi_G,\tau_\eta^*(\varphi))=1,
    \qquad
    \pfer(\varphi,\xi_G,\tau_\eta^*(\varphi))\geq\eta.
\end{align*}
\end{proposition}

Finite ARL ensures that the reporting time in \Cref{prop:impossibility} is almost surely finite; we give the proof in \Cref{app:proof-impossibility}. For example, a procedure whose first false alarm takes $10{,}000$ observations on average still has FDR one when we report at that alarm. Increasing the average wait does not change the proportion of false discoveries in this report. Nor does FDR control at each fixed time resolve this issue: choosing the reporting time from the data can change its error rate.

We propose to account for how long we have monitored when measuring Type I error. Dividing the error by expected monitoring time distinguishes an early false report from one that requires a long wait on average. This leads to the following metric, which we call \emph{error over patience} (EOP).

Let $\mathcal T_\P$ denote the set of $\F$-stopping times with finite expectation under $\P$. Since $\tau\geq1$, every $\tau\in\mathcal T_\P$ satisfies $1\leq\E[\tau]<\infty$ and is finite almost surely. We use this class in the EOP definition. The run-length lower bounds below also hold when the relevant expectation is infinite.

\begin{definition}[Error over patience]
For $Q\in\{\fdr,\fwer,\pfer\}$, define
\begin{align}
    \eop_Q(\varphi)=
    \sup_{\xi}\ \sup_{\tau\in\mathcal T_\P}
        \frac{Q(\varphi,\xi,\tau)}{\E[\tau]}.\label{eq:eop}
\end{align}
\end{definition}

We typically seek to upper bound EOP at a predefined level $\alpha>0$. One can think of small $\E[\tau]$ as low patience and large $\E[\tau]$ as high patience. More concretely, $\eop_Q(\varphi)\leq\alpha$ implies
\begin{align}
    Q(\varphi,\xi,\tau)\leq\alpha B
    \quad\text{whenever }\E[\tau]\leq B.\label{eq:budget}
\end{align}
For a chosen error budget $q$ and expected duration $B$, taking $\alpha=q/B$ therefore suffices. The denominator is an expectation: EOP bounds a ratio of expectations, and does not bound the realized false discovery proportion or $\E[\fdp_\tau/\tau]$. For FDR and FWER, \eqref{eq:budget} is informative when $\alpha B<1$.

Upper bounding $\eop_\fdr$ by $\alpha$ also gives $\arl_1(\varphi)\geq1/\alpha$: if the ARL is finite, apply \eqref{eq:eop} at $\tau_1^*(\varphi)$ under the global null, where the FDP equals one. EOP control asks for more than this run-length bound, since it applies to other stopping times and configurations. It is weaker than ordinary error control at level $\alpha$, since the permitted error grows with expected duration.

The linear denominator matches the expectation bound of an e-detector, which we define formally in \Cref{sec:edetcd}. For each error criterion, we construct a procedure whose bound has the form
\begin{align}
    Q(\varphi,\xi,\tau)\leq
    \frac{\alpha}{K}\sum_{k=1}^K
        \E[\tau\wedge(\xi^{(k)}-1)]\leq\alpha\E[\tau].\label{eq:patience}
\end{align}
Each summand counts how long stream $k$ is observed before its change, up to the stopping time. Thus, the sharper bound accounts for the time during which each stream can generate a false discovery. We use $\E[\tau]$ in the EOP definition because it refers to the duration of monitoring without requiring the changepoints.

Other discounts are possible. For a positive nondecreasing function $g:[1,\infty)\to(0,\infty)$, one could divide the error by $g(\E[\tau])$. If $Q\leq B_Q$ and $\eop_Q(\varphi)\leq\alpha$, then
\begin{align}
    \sup_\xi\sup_{\tau\in\mathcal T_\P}
    \frac{Q(\varphi,\xi,\tau)}{g(\E[\tau])}
    \leq\sup_{u\geq1}\frac{\min\{B_Q,\alpha u\}}{g(u)}.\label{eq:discount}
\end{align}
Here $B_Q=1$ for FDR and FWER, and $B_Q=K$ for PFER. For example, with $g(u)=u^p$, $0<p<1$, and $\alpha\leq B_Q$, the bound is $\alpha^p B_Q^{1-p}$. Thus, the linear EOP bound also gives guarantees for slower discounts. In \Cref{sec:weighted}, we construct procedures for concave discounts directly by changing the weights assigned to possible changepoints, and show how the weights affect detection of late changes.

\subsection{Run length when some streams change}\label{sec:partial}

We next show how EOP control gives run-length guarantees when some streams change. The global-null ARL only describes a system in which no stream changes. If we raise an alarm after declaring at least $\eta$ streams, we may instead want a long expected wait whenever fewer than $\eta$ streams actually change. This is the same reporting rule $\tau_\eta^*(\varphi)$ as in \Cref{prop:impossibility}: it counts all declarations and does not require knowledge of which are false.

Suppose at most $q$ streams ever change, where $q$ is an integer with $0\leq q<\eta$. At the alarm, at least $\eta-q$ declarations must then be false. If we bound how quickly these false discoveries can accumulate through EOP, we obtain a lower bound on the expected time to the alarm. Write
\begin{align*}
    \Xi_q=\{\xi:|\{k:\xi^{(k)}<\infty\}|\leq q\},
    \qquad 0\leq q<\eta.
\end{align*}
We prove the following bounds in \Cref{app:proof-partial}. They apply to any monitoring algorithm with the stated EOP guarantee.

\begin{proposition}\label{prop:partial}
For every configuration in $\Xi_q$,
\begin{align*}
    \E[\tau_\eta^*(\varphi)]&\geq\frac{1-q/\eta}{\alpha}
       &&\text{if }\eop_\fdr(\varphi)\leq\alpha,\\
    \E[\tau_\eta^*(\varphi)]&\geq\frac1\beta
       &&\text{if }\eop_\fwer(\varphi)\leq\beta,\\
    \E[\tau_\eta^*(\varphi)]&\geq\frac{\eta-q}{\beta}
       &&\text{if }\eop_\pfer(\varphi)\leq\beta.
\end{align*}
\end{proposition}

To interpret these bounds, suppose we wait for five declarations when at most two streams change. At least three declarations are false, and the false discovery proportion is at least $3/5$. FDR EOP control at level $\alpha$ therefore requires an expected wait of at least $3/(5\alpha)$; PFER EOP control at level $\beta$ requires at least $3/\beta$. FWER only records whether any false declaration occurs, so its bound is $1/\beta$.

Taking $q=0$ recovers the global-null guarantees. Taking $q=\eta-1$ covers every configuration with fewer than $\eta$ changes and gives lower bounds $1/(\alpha\eta)$ from FDR EOP and $1/\beta$ from either FWER or PFER EOP. The FDR bound decreases as more true changes are allowed because true declarations can reduce the false discovery proportion. The PFER bound instead depends on the number of false declarations needed to reach the alarm.

\section{Multiple testing with e-detectors}\label{sec:methods}

We now construct procedures that control EOP. We maintain one e-detector for each stream, using it to measure evidence that the stream has changed. We then apply a multiple testing rule to the $K$ detector values at each time. We first describe how to construct the individual detectors, then prove that the resulting reporting rules control false discoveries across streams and stopping times.

\subsection{Constructing e-detectors}\label{sec:edetcd}

We first give some background on e-values and e-processes. An \emph{e-value} for a null hypothesis is a nonnegative random variable with expectation at most one under every distribution in the null. An \emph{e-process} is a nonnegative process $(E_t)_{t\geq0}$, with each $E_t$ measurable with respect to $\F_t$, such that $\E[E_\tau]\leq1$ under every null distribution and every almost surely finite stopping time $\tau$. Thus, an e-process yields an e-value whenever it is stopped. One can think of it as accumulating evidence against the null; ideally, it grows quickly under an alternative. See \citet{ramdas2023game} for further background.

An \emph{e-detector} is a nonnegative process $(M_t)_{t\geq0}$, with each $M_t$ measurable with respect to $\F_t$, satisfying $M_0=0$ and $\E[M_\tau]\leq\E[\tau]$ for all stopping times with finite expectation under every null distribution \citep{shin2022detectors,ramdas2026universality}. Unlike an e-process, its null expectation is allowed to grow with monitoring time. Thresholding an e-detector at $b$ gives an alarm with ARL at least $b$ \citep{shin2022detectors}. The stopped expectation bound also applies when we choose to report for reasons other than this detector crossing its threshold. We use this property to combine evidence across streams.

For each stream $k$, we maintain an e-detector $M^{(k)}=(M_t^{(k)})_{t\geq0}$. A declaration about this stream can be false only before its changepoint, so we require its expectation bound up to the last pre-change observation:
\begin{align}
    \E\left[M_{\tau\wedge(\xi^{(k)}-1)}^{(k)}\right]
    \leq\E[\tau\wedge(\xi^{(k)}-1)]
    \quad\text{for every }k,\ \xi,\ \tau\in\mathcal T_\P.\label{eq:validity}
\end{align}
We use the convention $\infty-1=\infty$. The inequality must hold under each joint law in the model, including when other streams have changed. Under the global null, it is the usual e-detector condition. Requiring it for stopping times based on all streams is stronger than requiring it for stopping times based on stream $k$ alone.

We can construct such detectors from successive nonnegative, $\F_t$-measurable evidence factors $L_t^{(k)}$. We require that each factor have conditional expectation at most one while its stream remains pre-change:
\begin{align}
    \E[L_t^{(k)}\mid\F_{t-1}]\leq1
    \qquad\text{for }t<\xi^{(k)}.\label{eq:factor}
\end{align}
Starting both processes at zero, define the SR and CUSUM recursions
\begin{align}
    M_t^{\mathrm{SR},(k)}
       &=(M_{t-1}^{\mathrm{SR},(k)}+1)L_t^{(k)},&
    M_t^{\mathrm{CU},(k)}
       &=\max\{M_{t-1}^{\mathrm{CU},(k)},1\}L_t^{(k)}.\label{eq:recursions}
\end{align}
We prove in \Cref{sec:proofconstruction} that both recursions satisfy \eqref{eq:validity}. Classical SR and CUSUM statistics are specific instances: when the pre-change and post-change observations have known densities $p_0$ and $p_1$, and the conditional pre-change density given the global past is $p_0$, we can take $L_t=p_1(X_t)/p_0(X_t)$ whenever this ratio is well-defined. We can also choose factors satisfying \eqref{eq:factor} without specifying either density, as in the conditional-symmetry example in \Cref{app:filtration}. The same recursions then give nonparametric e-detectors.

To see how these recursions account for an unknown changepoint, we write them as
\begin{align}
    M_t^{\mathrm{SR},(k)}
       =\sum_{j=1}^t\prod_{s=j}^t L_s^{(k)},
    \qquad
    M_t^{\mathrm{CU},(k)}
       =\max_{1\leq j\leq t}\prod_{s=j}^t L_s^{(k)}.\label{eq:restart}
\end{align}
Each product measures evidence from observations $j,\ldots,t$, treating $j$ as a possible changepoint. SR sums the evidence over all possible starting times, while CUSUM keeps the largest value. In both cases, we start a new product at each observation, so a late change need not overcome all the evidence accumulated before it. We prove the error bounds below for any collection of processes satisfying \eqref{eq:validity}; the two recursions give concrete choices for the experiments and delay analysis.

\subsection{The e-d-BH and e-d-Bonferroni procedures}

Given one detector $M^{(k)}$ for each stream, we must decide which streams have enough evidence to declare a change. We make this decision using all $K$ detector values together. At time $t$, let $M_t^{[1]}\geq\cdots\geq M_t^{[K]}$ be these values in decreasing order. Given a level $\alpha\in(0,1)$, the \emph{e-d-BH procedure} sets
\begin{align}
    r_t=\max\left\{r\in[K]:M_t^{[r]}\geq\frac{K}{\alpha r}\right\},\label{eq:edbh}
\end{align}
with $r_t=0$ when the set is empty, and declares the $r_t$ streams with largest detector values. We use a fixed ordering to break ties. This applies the e-BH rule \citep{wang2022false} at each time, although an unnormalized e-detector value need not itself be an e-value.

For a level $\beta>0$, the \emph{e-d-Bonferroni procedure} declares
\begin{align}
    R_t^{\mathrm{Bonf}}=\{k:M_t^{(k)}\geq K/\beta\}.\label{eq:edbonf}
\end{align}
With SR inputs, this is precisely a collection of parallel SR charts at the common threshold $K/\beta$. The same observation holds for CUSUM inputs. Below, we prove a guarantee for the false declarations made by these charts across streams and stopping times.

The e-d-BH threshold depends on the number of declarations. If several streams have large detector values, the threshold for including another stream can be smaller than $K/\alpha$. At equal levels $\beta=\alpha$, e-d-BH includes every declaration made by e-d-Bonferroni. We write $\varphi^{\mathrm{BH}}$ and $\varphi^{\mathrm{Bonf}}$ for the two monitoring algorithms.

In our main theorem, we prove that e-d-BH controls FDR EOP and e-d-Bonferroni controls PFER EOP. We allow each stream to change at a different time, including never, and the reporting time may depend on all the observations. We do not require independence across streams; we require each detector to satisfy \eqref{eq:validity} with respect to the global filtration.

\Needspace{14\baselineskip}
\begin{samepage}
\begin{theorem}[EOP control]\label{thm:eop}
Suppose each $M^{(k)}$ satisfies \eqref{eq:validity}. For every configuration and $\tau\in\mathcal T_\P$, e-d-BH satisfies
\begin{align}
    \fdr(\varphi^{\mathrm{BH}},\xi,\tau)
    \leq\frac{\alpha}{K}\sum_{k=1}^K
          \E[\tau\wedge(\xi^{(k)}-1)],\label{eq:bhbound}
\end{align}
and e-d-Bonferroni satisfies
\begin{align}
    \pfer(\varphi^{\mathrm{Bonf}},\xi,\tau)
    \leq\frac{\beta}{K}\sum_{k=1}^K
          \E[\tau\wedge(\xi^{(k)}-1)].\label{eq:bonfbound}
\end{align}
In particular, $\eop_\fdr(\varphi^{\mathrm{BH}})\leq\alpha$ and
$\eop_\pfer(\varphi^{\mathrm{Bonf}})\leq\beta$. Since FWER is bounded by PFER, we also have $\eop_\fwer(\varphi^{\mathrm{Bonf}})\leq\beta$.
\end{theorem}
\end{samepage}

We prove \Cref{thm:eop} in \Cref{app:proof-eop}. The main idea is to bound the contribution of each false declaration by its stream's detector value. The BH threshold accounts for the total number of declarations when bounding their false proportion; the Bonferroni threshold instead bounds the number of false declarations. We then use the stopped expectation bound for each still-null stream and sum over streams.

The bounds retain more information than the EOP guarantee alone. A stream can contribute false discoveries only while it remains pre-change, which explains the term $\E[\tau\wedge(\xi^{(k)}-1)]$. Under the global null, every term equals $\E[\tau]$, giving FDR at most $\alpha\E[\tau]$ for e-d-BH and PFER at most $\beta\E[\tau]$ for e-d-Bonferroni. When some streams change, their contributions stop at their changepoints. The procedure does not need to know these times to obtain the guarantee.

For example, if we will report by a horizon $H$ and want FDR at most $q$, we can run e-d-BH at level $\alpha=q/H$. We may stop at the first declaration, after several declarations, or at another time chosen from the observations, provided it is at most $H$. More generally, the same conclusion holds for any stopping rule with expected duration at most $H$. We use the bounded-horizon choice in the experiments.

We can also express the theorem as a lower bound on the expected time to false declarations. This connects the EOP guarantee to the usual ARL criterion and makes its consequences under partial changes explicit.

\Needspace{9\baselineskip}
\begin{samepage}
\begin{corollary}[Run-length bounds]\label{cor:arl}
Under the global null, e-d-BH satisfies $\arl_1(\varphi^{\mathrm{BH}})\geq1/\alpha$. Under any configuration, e-d-Bonferroni satisfies
\begin{align*}
    \arl_\eta(\varphi^{\mathrm{Bonf}},\xi)\geq\eta/\beta.
\end{align*}
For e-d-BH under a general configuration, we have
$\arl_1(\varphi^{\mathrm{BH}},\xi)\geq1/(K\alpha)$.
If at most $q$ streams ever change, the latter bound improves to $1/\{(q+1)\alpha\}$.
\end{corollary}
\end{samepage}

Under the global null, the first declaration is necessarily false, so its false discovery proportion is one. If some streams have changed, the first false declaration can appear alongside true declarations. With at most $q$ true changes, that report has false discovery proportion at least $1/(q+1)$. Thus, a bound on FDR permits a shorter expected wait to the first false declaration in this case. This also agrees with the BH rule: true discoveries can lower the threshold for declaring another stream.

For e-d-Bonferroni, we bound the expected number of false declarations directly. Reaching $\eta$ false declarations therefore requires an expected wait of at least $\eta/\beta$, regardless of how many true changes have occurred. We prove these claims in \Cref{app:proof-arl}. The bounds concern the time to false declarations, which differs from the time to $\eta$ declarations of either kind in \Cref{prop:partial}.

\subsection{Fixed-time error control}\label{sec:fixed}

The EOP guarantee allows the error to grow with expected monitoring time. If we instead want FDR at most $\alpha$ at each fixed time, we can normalize the detectors before applying e-BH. For a stream that has not changed by time $t$, \eqref{eq:validity} gives $\E[M_t^{(k)}]\leq t$. Dividing by $t$ therefore produces an e-value. We apply the usual multiple testing procedures to these normalized values.

For normalized e-d-BH, we replace the threshold $K/(\alpha r)$ by $Kt/(\alpha r)$, or equivalently use level $\alpha_t=\alpha/t$ in \eqref{eq:edbh}. For normalized e-d-Bonferroni, we replace $K/\beta$ by $Kt/\beta$. With SR inputs and a fixed collection of streams, the normalized BH rule gives the basic SHEREF procedure of \citet{lee2026sequential} as a special case.

\begin{proposition}[Fixed-time control]\label{prop:fixed}
Suppose $\E[M_t^{(k)}]\leq t$ whenever $t<\xi^{(k)}$. Applying e-BH at level $\alpha$ to $(M_t^{(k)}/t)_{k=1}^K$ gives
\begin{align}
    \fdr(\varphi,\xi,t)\leq
        \alpha\frac{|\mathcal H_0(t)|}{K}\leq\alpha
    \qquad\text{for every fixed }t.\label{eq:fixedbh}
\end{align}
Applying e-Bonferroni at level $\beta$ to the same values gives PFER at most $\beta|\mathcal H_0(t)|/K$. If the full stopped expectation bound \eqref{eq:validity} also holds, then the normalized BH rule has infinite global-null ARL for $\alpha<1$.
\end{proposition}

We prove \Cref{prop:fixed} in \Cref{app:proof-normalized}. The fixed-time error bounds require only the expectation bound at the time of the report. The infinite-ARL conclusion uses the stronger requirement that the detector bound hold at stopping times. We distinguish these assumptions because a fixed-time FDR guarantee alone does not apply at the first declaration.

Normalization makes the evidence threshold increase with time. This differs from the constant thresholds used for EOP control, which continue to give the same opportunity to detect a change even after a long period without changes. On a finite horizon $H$, the two calibrations can also be compared directly. At a common error target $q$, the normalized BH threshold is $Kt/(qr)$, whereas e-d-BH with EOP level $q/H$ uses $KH/(qr)$. The normalized rule thus has a smaller threshold before $H$, in exchange for a guarantee at each fixed reporting time. We compare both rules in \Cref{app:internal-comparison}.

Infinite ARL is compatible with \Cref{prop:impossibility}, whose premise is finite ARL. It does not by itself bound the probability of ever making a false declaration. We next obtain this additional guarantee by changing the evidence processes themselves.

\subsection{Using e-processes}\label{sec:eprocesses}

A second option is to require FDR or PFER control at every stopping time, without dividing by expected monitoring time. To do so, we maintain an e-process $E^{(k)}$ in each stream in place of its e-detector. We require
\begin{align}
    \E\left[E_{\tau\wedge(\xi^{(k)}-1)}^{(k)}\right]\leq1
    \quad\text{for every }k,\ \xi,\ \tau\in\mathcal T_\P.\label{eq:eprocessvalid}
\end{align}
The expectation bound is now one, regardless of how long we monitor. We apply the same e-BH and e-Bonferroni rules as before, using $E_t^{(k)}$ instead of $M_t^{(k)}$ and the constant thresholds $K/(\alpha r)$ and $K/\beta$.

\begin{proposition}[Stopping-time control with e-processes]\label{prop:eprocesses}
Under \eqref{eq:eprocessvalid}, e-BH has FDR at most $\alpha$ and e-Bonferroni has PFER at most $\beta$ for every configuration and $\tau\in\mathcal T_\P$. Under the global null, the probability of ever making a declaration is at most $\alpha$ for e-BH and at most $\beta$ for e-Bonferroni. Their global-null ARLs are therefore infinite when $\alpha<1$ and $\beta<1$, respectively.
\end{proposition}

We prove \Cref{prop:eprocesses} in \Cref{app:proof-eprocesses}. At levels below one, these procedures must have positive probability of never declaring a change under the global null. This is how they provide ordinary error control even when we report at the first declaration. In contrast, EOP control permits eventual false declarations while accounting for the expected time until the report.

We can construct suitable e-processes by weighting evidence from possible changepoints. For each candidate time $j$, start with value one before $j$ and multiply the factors $L_s^{(k)}$ from time $j$ onwards. This gives evidence for a change starting at $j$. Since the actual changepoint is unknown, we combine these processes with nonnegative deterministic weights $(w_j)_{j\geq1}$ satisfying $\sum_jw_j\leq1$:
\begin{align}
    E_t^{(k)}
    =\sum_{j=1}^t w_j\prod_{s=j}^tL_s^{(k)}
      +\sum_{j>t}w_j.\label{eq:weighted-eprocess}
\end{align}
The first sum combines the products already started. The second accounts for candidate changepoints after $t$, whose processes still have value one. Under \eqref{eq:factor}, each component has nonincreasing conditional expectation before the stream changes. The bound on the total weight then gives \eqref{eq:eprocessvalid}; we give the proof in \Cref{app:proof-eprocesses}.

For example, $w_j=1/\{j(j+1)\}$ gives positive weight to every candidate changepoint and sums to one. The remaining sum in \eqref{eq:weighted-eprocess} is then $1/(t+1)$. Every possible starting time contributes to the process, but evidence from a late starting time receives less weight. For a change at $\nu$ with $w_\nu>0$, it suffices for the product starting there to reach $K/(\alpha r w_\nu)$ for that component alone to meet the BH threshold for $r$ declarations. The corresponding sufficient threshold for an SR e-detector is $K/(\alpha r)$. In this sufficient condition, weighting adds $\log(1/w_\nu)$ to the log threshold for this product, equal to $\log\{\nu(\nu+1)\}$ for the example above. Other components of the mixture can also contribute to a declaration.

SR instead assigns weight one to every possible starting time. Its null expectation is allowed to grow with time, so a late change receives the same initial weight as an early one. We thus obtain different error guarantees by choosing how to combine the same evidence products. In the next section, we study how the constant-threshold e-detector procedures translate evidence after a change into detection delay.

\section{Power and detection delay for Gaussian changes}\label{sec:power}

To understand the detection delay of our algorithms, we consider a simple Gaussian mean-change setting. We use it to relate the time needed to detect changes to their size, the number of streams, and the number of changes we want to identify. We first prove finite-sample power and delay bounds, then prove that both procedures attain the smallest possible leading delay term as the error level decreases. Suppose the observation vectors are independent over time, with
\begin{align}
    X_t^{(k)}\sim
    \begin{cases}
        \mathcal N(0,1),&t<\xi^{(k)},\\
        \mathcal N(\mu,1),&t\geq\xi^{(k)},
    \end{cases}
    \qquad \mu>0.\label{eq:gaussianmodel}
\end{align}
The streams may be dependent. We take the shift $\mu$ to be known and use
$L_t^{(k)}=\exp\{\mu X_t^{(k)}-\mu^2/2\}$
in either recursion in \eqref{eq:recursions}. The conditional expectation in \eqref{eq:factor} is one before the change, so all the error bounds above apply. Write $I=\mu^2/2$ for the mean post-change log likelihood ratio. A likelihood product starting at the change therefore accumulates mean log evidence $nI$ after $n$ observations. Comparing this growth with the logarithm of the detection threshold explains the dependence on $\mu$ and $\alpha$ in the bounds below.

\subsection{Power and delay}

Suppose a set $S$ of $m\in[K]$ streams changes at time $\nu$, while the remaining streams do not change. Define
\begin{align}
    T_n=\sum_{k\in S}\varphi_{\nu+n-1}^{(k)},\label{eq:powerdefs}
\end{align}
the number of changed streams declared after $n$ post-change observations. We want at least $r$ of the $m$ changed streams to be declared, with probability at least $1-\delta$. The following theorem gives a sufficient number of observations for this to happen, conditional on the observations before the change.

\begin{theorem}[Detecting changed streams]\label{thm:power}
Consider e-d-BH at level $\alpha$ with SR or CUSUM inputs under \eqref{eq:gaussianmodel}. For any $r\in[m]$ and $\delta\in(0,1)$, we have
$\P(T_n\geq r\mid\F_{\nu-1})\geq1-\delta$ whenever
\begin{align}
    n\geq
    \frac{4}{\mu^2}\log\frac{K}{\alpha r}
    +\frac{8}{\mu^2}\log\frac{m}{\delta(m-r+1)}.
    \label{eq:power-rate}
\end{align}
\end{theorem}

The first term reflects the evidence threshold $K/(\alpha r)$: if at least $r$ changed streams reach it, BH declares them. The second term allows for variation in the accumulated evidence and the desired probability of detection. Increasing $r$ lowers the BH threshold, but also requires more changed streams to have accumulated enough evidence. Requiring a smaller failure probability $\delta$ increases the bound only logarithmically.

To detect a fraction $c\in(0,1)$ of the changed streams, take $r=\lceil cm\rceil$. The second logarithm is then at most $\log\{1/[\delta(1-c)]\}$. For fixed $c$ and $\delta$, a number of post-change observations of order $[1+\log\{K/(\alpha m)\}]/\mu^2$ therefore suffices. The logarithmic term decreases when more streams change, since more streams can contribute to lowering the BH threshold. The bound grows as $1/\mu^2$ when the shift becomes smaller.

If we require all $m$ changed streams to be declared, we take $r=m$, which replaces the second logarithm by $\log(m/\delta)$. The additional dependence on $m$ comes from requiring that none of the changed streams be missed. The reported set may also contain unchanged streams; its FDR remains bounded by \Cref{thm:eop}.

The proof in \Cref{app:proof-power} bounds the probability that a changed stream has insufficient evidence, then bounds the number of streams that fall below the threshold. For independent streams, we give sharper bounds in \Cref{app:further-power}, including a bound on the \emph{true discovery rate} $\E[T_n/m]$, the expected proportion of changed streams declared.

The delay until at least $r$ changed streams are declared is
\begin{align*}
    D_r=\inf\{n\geq1:T_n\geq r\},\qquad r\in[m].
\end{align*}
In \Cref{app:proof-delay}, we prove an expected-delay bound of the same order as \eqref{eq:power-rate}. This bound does not increase for late changes: both SR and CUSUM include a likelihood product starting at each possible changepoint, so evidence accumulated from $\nu$ onwards gives the same bound regardless of $\nu$. Since $D_r$ depends on the unknown changed streams and changepoint, we use it to assess detection performance.

For e-d-Bonferroni at level $0<\beta<K$, we replace $K/(\alpha r)$ in \eqref{eq:power-rate} by $K/\beta$. At equal levels, its sufficient evidence threshold stays fixed as the number of changed streams increases. The lower threshold for e-d-BH suggests that it can identify several changes sooner. We compare the actual delays and false discoveries in the experiments, since the two procedures control different errors and the upper bounds alone do not order their delays.

These bounds describe a monitoring rule that continues after earlier declarations. If monitoring ends at the first alarm, an alarm before the change also matters. We report its probability separately in the experiments.

\subsection{Asymptotic optimality}

We now prove a matching first-order lower bound on delay as the error level decreases. We compare with every monitoring algorithm satisfying FDR EOP control, including algorithms constructed without e-detectors. For this result, we assume that the streams are also independent. Let $\mathcal A_\alpha$ be the class of monitoring algorithms with $\eop_\fdr\leq\alpha$ over every configuration in the Gaussian model \eqref{eq:gaussianmodel}. In particular, the error bound must hold when fewer than $m$ streams change.

Write $\P_{S,\nu}$ for the law under which precisely the streams in $S$ change at time $\nu$, and write $D_r(\varphi;S,\nu)$ to make the dependence of the delay on the procedure and configuration explicit. We use the worst conditional expected delay, following the minimax formulation of \citet{lorden1971procedures}:
\begin{align}
    \mathcal D_{m,r}(\varphi)
    =\sup_{|S|=m}\sup_{\nu\geq1}
      \operatorname*{ess\,sup}
      \E_{S,\nu}[D_r(\varphi;S,\nu)\mid\F_{\nu-1}].
    \label{eq:worst-delay}
\end{align}
Thus we consider every changed set of size $m$ and every changepoint. We also condition on the observations before the change and take the largest conditional expected delay, apart from histories with probability zero. With this criterion, we require a procedure to remain quick even after an unfavorable history or a late change. The infimum in the next theorem chooses the best procedure subject to the error constraint. As above, monitoring continues after earlier declarations.

\Needspace{10\baselineskip}
\begin{samepage}
\begin{theorem}[Gaussian optimality]\label{thm:gaussian-optimality}
Fix $K$, $1\leq r\leq m\leq K$, and $\mu>0$. For independent streams in \eqref{eq:gaussianmodel},
\begin{align}
    \inf_{\varphi\in\mathcal A_\alpha}\mathcal D_{m,r}(\varphi)
    \sim\frac{2\log(1/\alpha)}{\mu^2},
    \qquad\alpha\downarrow0.
    \label{eq:gaussian-minimax}
\end{align}
Both e-d-BH at level $\alpha$ and e-d-Bonferroni at level $\beta=\alpha$, with either SR or CUSUM inputs, attain this bound.
\end{theorem}
\end{samepage}

We prove \Cref{thm:gaussian-optimality} in \Cref{app:proof-optimality}. The lower bound compares a configuration in which stream $k$ changes with one in which only that change is removed. The evidence from the other streams has the same distribution in both cases, so declaring stream $k$ still requires sufficient evidence from that stream. This is where identification differs from raising a global alarm, for which several changed streams can contribute evidence about the same event.

The absence of $K,m,r$ from the leading term comes from holding them fixed while $\alpha$ tends to zero. For example, the logarithm of the BH threshold is
\begin{align*}
    \log\frac{K}{\alpha r}=\log\frac1\alpha+\log\frac Kr.
\end{align*}
The first term grows, while the second stays constant. The corresponding logarithm for Bonferroni is $\log(1/\alpha)+\log K$, so both have the same leading term. Requiring $r$ of a fixed number $m$ of changed streams to cross their thresholds also affects terms below this leading order. The shift size remains in the limit because it determines the rate $I=\mu^2/2$ at which each changed stream accumulates log evidence.

We therefore obtain a sharp first-order delay bound under the EOP constraint, and prove that both proposed procedures attain it. Thus EOP control permits asymptotically optimal identification of changed streams in this Gaussian setting. The dependence on $K,m,r$ is still present in the finite-sample bounds, and differences between BH and Bonferroni can matter at the levels used in practice. To compare those differences, or settings with a growing number of streams, we turn to the experiments.

% External comparisons, followed by comparisons within our proposed family.
% Keep the main figures with their discussion without stretching page spacing.
\begingroup
\raggedbottom
\section{Experiments}\label{sec:sims}

We compare our procedures with existing methods for identifying changed streams, then study the time needed to identify a substantial fraction of the changes. We also examine the effect of the BH threshold with non-Gaussian observations. We use our EOP bounds to calibrate the procedures and assess the resulting error and identification at specified reporting times. Finally, we compare global detection delays, for methods whose aim is to raise an alarm without identifying the changed streams. All code for our experiments can be found at the following GitHub repository:

\begin{center}
  \url{https://github.com/sanjitdp/multistream-change-detectors}
\end{center}

\subsection{Comparison with existing identification methods}\label{sec:external-identification}

We simulate $K=100$ independent Gaussian streams with variance one and shifts of size $\mu=0.5$. We use a horizon $H=400$ and an error budget $q=0.05$. For both e-d-BH and e-d-Bonferroni, we take $\alpha=q/H$ and use SR statistics. This gives thresholds $800{,}000/r$ for e-d-BH at rank $r$, and $800{,}000$ for e-d-Bonferroni. Our EOP bounds then give FDR control for e-d-BH and PFER control for e-d-Bonferroni at any reporting time bounded by $H$.

We compare with the CUSUM procedure of \citet{gandy2013nonrestarting}, MD-FDR of \citet{chen2020false}, and the posterior LFDR procedure of \citet{lu2022optimal}. Gandy--Lau applies ordinary BH to p-values computed from the null distribution of bounded CUSUM statistics. We use its fixed-time level $q$ and also a time-adjusted level $q/H$, which gives FDR at most $q$ at any reporting time up to $H$. The Bayesian procedures use independent changepoint priors with probability $1/2$ of no change and a geometric distribution of mean 100 conditional on a change. Their first-report FDR guarantees average over this prior. We give the exact procedures and parameter choices in \Cref{app:external-details}.

For each method, we stop at its first nonempty declaration, or report an empty set if none is made by $H$. We record FDR and the number of true discoveries in that report. This is a common decision for all the methods: the Bayesian procedures have not removed any streams before their first report. Each setting uses 3000 paired runs. We consider no changes, 10 or 50 simultaneous changes at time 101 or 251, and changes generated from the Bayesian prior.

\begin{table}[!htbp]
\centering
\caption{Error and identification at the first nonempty report through time 400 (an empty set if there is no report). There are 100 streams; in the two alternative settings, 50 change by a mean of 0.5 at the indicated time. True counts are unconditional and are zero for pre-change or absent reports. Entries are estimates (Monte Carlo standard errors) from 3000 paired runs. Smaller FDR $\downarrow$ and larger true count $\uparrow$ are preferable. The methods have different stated guarantees; see the text. With $K=100$ and $\alpha=0.05/400$, our SR procedures use $M_t^{[r]}\geq800{,}000/r$ for e-d-BH (rank $r$) and $M_t^{(k)}\geq800{,}000$ for e-d-Bonferroni.}
\label{tab:external-main}
\begingroup\fontsize{9}{11}\selectfont
\setlength{\tabcolsep}{3pt}\renewcommand{\arraystretch}{1.12}
\begin{tabular}{@{}lrrrrr@{}}
\hline
 & No changes & \multicolumn{2}{c}{Change at 101} & \multicolumn{2}{c}{Change at 251} \\
Method & FDR $\downarrow$ & FDR $\downarrow$ & True count $\uparrow$ & FDR $\downarrow$ & True count $\uparrow$ \\
\hline
e-d-BH (ours) & .0267 (.0029) & .0031 (.0010) & 2.050 (.030) & .0207 (.0026) & 1.982 (.031) \\
e-d-Bonferroni (ours) & .0267 (.0029) & .0032 (.0010) & 1.071 (.005) & .0203 (.0026) & 1.056 (.006) \\
Gandy--Lau & .9060 (.0053) & .6200 (.0088) & 1.181 (.044) & .8156 (.0070) & 0.609 (.034) \\
Gandy--Lau (time adj.) & .0037 (.0011) & .0047 (.0012) & 1.752 (.024) & .0047 (.0012) & 1.822 (.026) \\
MD-FDR & .0287 (.0030) & .0085 (.0016) & 2.014 (.030) & .0251 (.0029) & 1.844 (.027) \\
Posterior LFDR & .9487 (.0040) & .6948 (.0084) & 0.400 (.012) & .9192 (.0050) & 0.102 (.007) \\
\hline
\end{tabular}
\endgroup
\end{table}

In Table~\ref{tab:external-main}, we find that the error requirement has a substantial effect on the first report. The fixed-time Gandy--Lau guarantee does not protect a report made at the first alarm; we observe much larger FDR at that time. The time adjustment prevents this accumulation of error. Against this adjusted rule, we find that e-d-BH reports more true changes and reports sooner in each fixed-change setting. With 50 changes, its mean post-change report delay is about 29 observations, compared with 38--39 for the adjusted Gandy--Lau rule. These delays condition on no earlier report; we give the pre-change report probabilities and full delay comparisons in \Cref{app:external-details}. Our EOP calibration thus provides a useful error bound for data-dependent reporting. We evaluate its consequences at the specified reporting rule, without estimating the supremum in EOP.

We find that MD-FDR reports sooner than e-d-BH for changes at time 101, but later for changes at time 251. Its prior assigns less probability to a new change as time passes, whereas our restart weights do not decrease. e-d-BH has the largest estimated mean number of true discoveries in the first report among the procedures considered in each fixed-change setting, although it does not always have the smallest FDR. When we generate changes from the stated prior, the Bayesian procedures benefit from this information: posterior LFDR reports much sooner while keeping prior-averaged FDR below $q$. We include that setting and the sparse alternatives in \Cref{app:external-details}.

\FloatBarrier
\subsection{Time to identify several changes}\label{sec:identification-delay}

A first report may contain only a few true discoveries. To study how long it takes to identify a substantial fraction of the changes, we continue monitoring after earlier declarations and record the first time at least 80\% of the changed streams are declared together. We again use $K=100$ independent Gaussian streams, shifts of size $0.5$, $H=400$ and $q=0.05$, with either 10 or 50 streams changing at time 101. Each setting uses 4000 paired runs. The identification times depend on the true changed set and serve only to assess performance.

We use the same SR inputs and level $\alpha=q/H$ for e-d-BH and e-d-Bonferroni. We also include separate SR charts at threshold $1/\alpha=8000$, which adjust for monitoring time but not the number of streams, and SHEREF at fixed-time FDR level $q$. Table~\ref{tab:tradeoff} reports the mean number of post-change observations needed to reach the identification target, together with the probability of any declaration by $H$ under the global null. We count the first post-change observation as one; all runs reached the target by $H$.

\begin{table}[!htbp]
\centering\small
\caption{Gaussian comparison with $K=100$, $H=400$ and $q=0.05$, using 4000 paired runs per setting. Changes have size $0.5$ and occur at time 101. Null alarm probability is the probability of any declaration by $H$ when no stream changes. At descending rank $r$, e-d-BH uses the SR threshold $800{,}000/r$; e-d-Bonferroni and the separate SR charts use $800{,}000$ and $8000$, respectively. SHEREF uses $2000t/r$ at time $t$. Identification columns give the mean number of post-change observations needed to declare 80\% of the changed streams. Parentheses give Monte Carlo standard errors.}
\label{tab:tradeoff}
\begin{tabular}{lrrr}
\toprule
 & Null alarm & \multicolumn{2}{c}{Identification delay $\downarrow$} \\
Method & probability $\downarrow$ & 10 changes & 50 changes \\
\midrule
e-d-BH (ours) & $0.031\;(0.003)$ & $84.7\;(0.2)$ & $76.9\;(0.1)$ \\
e-d-Bonferroni (ours) & $0.031\;(0.003)$ & $103.8\;(0.3)$ & $112.0\;(0.1)$ \\
Per-stream SR & $0.970\;(0.003)$ & $61.3\;(0.2)$ & $67.9\;(0.1)$ \\
SHEREF & $0.059\;(0.004)$ & $77.1\;(0.2)$ & $68.3\;(0.1)$ \\
\bottomrule
\end{tabular}
\end{table}

We find that e-d-BH identifies 80\% of the changes sooner than e-d-Bonferroni, with a larger gain when more streams change. Both procedures have the same estimated global-null alarm probability in these experiments. With more changed streams, BH can lower its threshold further, whereas Bonferroni continues to require the same evidence in each stream. This agrees with the threshold dependence in \Cref{sec:power}, although the two procedures control different errors: FDR for e-d-BH and PFER for e-d-Bonferroni.

The separate SR charts identify changes quickly but make a false declaration in nearly every global-null run. SHEREF also identifies changes sooner than e-d-BH, using a smaller threshold before $H$, but its fixed-time guarantee does not cover stopping at the first declaration. We therefore assess identification time alongside error at the reporting rule of interest. In \Cref{app:internal-comparison}, we give the full identification curves and further comparisons of reporting times.

\FloatBarrier
\subsection{Non-Gaussian observations}
\label{sec:non-gaussian-experiments}

We next compare our procedures when the evidence does not come from an independent Gaussian model. We consider heavy-tailed observations with dependence across streams, and Poisson counts with dependence over time. In both settings, we apply e-d-BH to SR statistics and e-d-Bonferroni to both SR and CUSUM statistics, using the same factors and $\alpha=0.05/400$. Thus our thresholds are $8000K/r$ for e-d-BH at rank $r$, and $8000K$ for e-d-Bonferroni. Our bounds give FDR control for e-d-BH and PFER control for e-d-Bonferroni at level 0.05 at any reporting time bounded by 400.

We report ordinary FDR and the expected proportion of already changed streams currently declared, which we call the true discovery rate (TDR). Each setting uses 1500 runs. We also include SHEREF, which applies e-BH to the SR statistics divided by the current time, at fixed-time FDR level 0.05. Its threshold $20Kt/r$ is lower than e-d-BH's before time 400, and the two declaration sets agree at time 400. We give the complete data-generating models in \Cref{app:diverse-details}.

\subsubsection{Heavy-tailed observations}

For the heavy-tailed observations, we combine shared and stream-specific Student-$t$ noise with three degrees of freedom. We consider independent streams and streams with correlation $\rho=0.5$. The changed streams have shifts of size 0.45 or 0.9. We use the bounded factors $L_t^{(k)}=1+0.5\operatorname{sign}(X_t^{(k)})$, which satisfy \eqref{eq:factor} by conditional symmetry. These factors use the signs of observations, so a large outlier cannot by itself contribute arbitrarily large evidence.

In Figure~\ref{fig:diverse-heavy-full}, we change 20 of the 100 streams at time 101.

\begin{figure}[H]
\centering
\includegraphics[width=\textwidth]{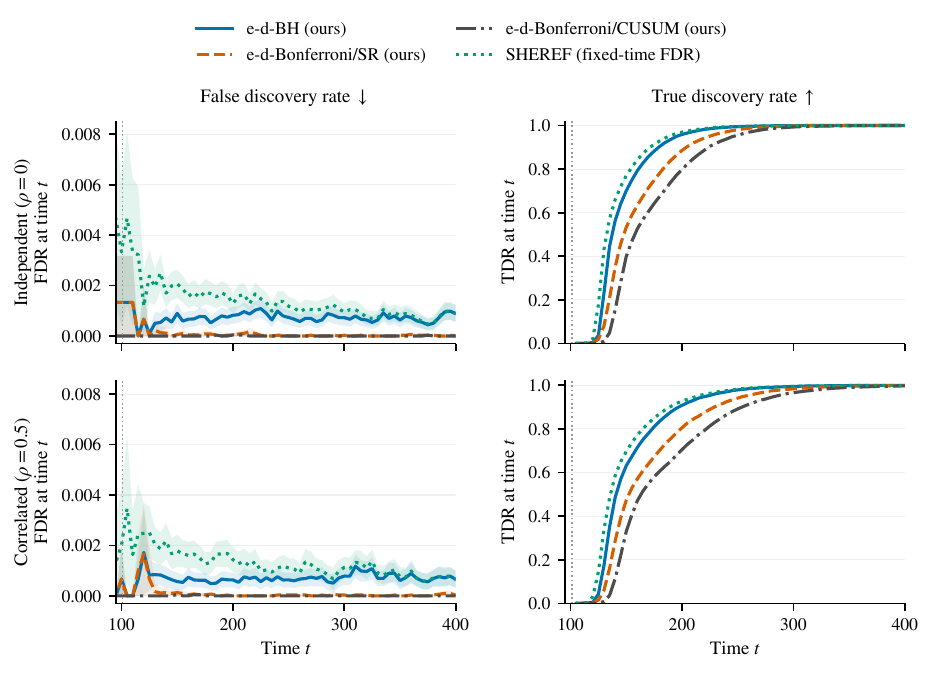}
\caption{FDR $\downarrow$ and TDR $\uparrow$ for simultaneous heterogeneous shifts in 20 of 100 heavy-tailed streams at time 101. Rows use independent and correlated streams. Our thresholds are $800{,}000/r$ for e-d-BH at rank $r$, and $800{,}000$ for both e-d-Bonferroni procedures, with $\alpha=0.05/400$. SHEREF uses threshold $2000t/r$, for fixed-time FDR level 0.05. Shading shows 1.96 Monte Carlo standard errors over 1500 runs.}
\label{fig:diverse-heavy-full}
\end{figure}

We find higher identification rates for e-d-BH than for either e-d-Bonferroni procedure. Evidence from the larger shifts allows e-d-BH to lower its threshold and identify more of the smaller shifts. We observe this gain both for independent and correlated streams, even though the input processes use only signs. We find that all four procedures identify nearly all changed streams by the end of monitoring; the main difference is how much earlier they become identifiable. Using the same SR inputs for e-d-BH and e-d-Bonferroni isolates the effect of their different thresholds.

We also find that e-d-Bonferroni makes fewer false discoveries, in keeping with its stronger PFER requirement. The FDR curves remain well below 0.05 for all four procedures in this experiment. SHEREF identifies changes earlier than e-d-BH because its threshold is smaller before time 400. At time 400, their thresholds and declaration sets agree. Thus the earlier identification from SHEREF accompanies a different error requirement: its guarantee applies at each fixed time, while our calibration also covers data-dependent reporting.

We also consider changes at different times. When ten streams change at time 101 and another ten at time 201, we find that e-d-BH retains its identification advantage after the second group changes. We give this experiment and its interpretation in Figure~\ref{fig:diverse-staggered-full} in \Cref{app:diverse-details}.

\subsubsection{Dependent Poisson counts}

We next consider observations with dependence over time. Before a change, stream $k$ has a Poisson distribution conditional on the past, with mean $a_k+0.3X_{t-1}^{(k)}$, where the known intercepts alternate between $a_k=0.1$ and $a_k=1$. A change multiplies this conditional mean by 1.75. The streams are independent of one another, and we use the exact conditional likelihood ratios to construct the e-detectors. At each $K\in\{20,100,500\}$, either 10\% or 50\% of the streams change at time 101. Figure~\ref{fig:diverse-counts-full} reports the error and identification rates at time 200. We give the initialization and likelihood ratios in \Cref{app:diverse-details}.

\begin{figure}[H]
\centering
\includegraphics[width=\textwidth]{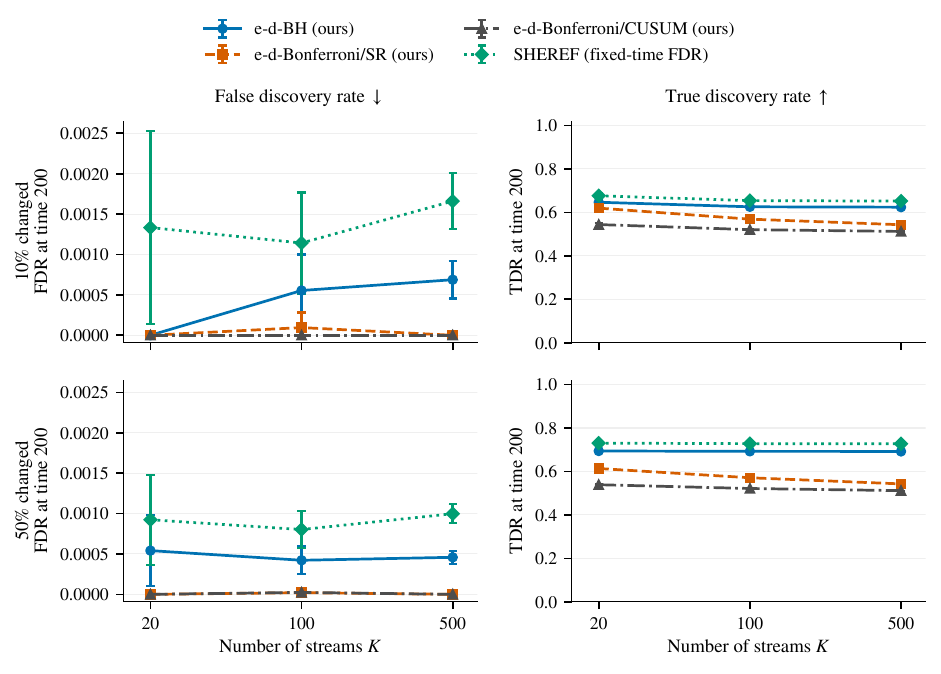}
\caption{FDR $\downarrow$ and TDR $\uparrow$ at time 200 for Poisson counts whose conditional means depend on the previous count. Either 10\% or 50\% of streams change at time 101. Our thresholds are $8000K/r$ for e-d-BH at rank $r$, and $8000K$ for both e-d-Bonferroni procedures, with $\alpha=0.05/400$. SHEREF uses threshold $4000K/r$ at time 200. The number of streams is on a logarithmic scale. Bars show 1.96 Monte Carlo standard errors over 1500 runs; some lie inside the markers.}
\label{fig:diverse-counts-full}
\end{figure}

When half the streams change, we find that e-d-BH's identification rate remains nearly constant as $K$ grows. With 10\% changed, it decreases modestly, while the e-d-Bonferroni rates decrease more. Here the number of changed streams grows with $K$. If the number of declarations is proportional to $K$, the step-up threshold $K/(r\alpha)$ remains of constant order, while the e-d-Bonferroni threshold $K/\alpha$ grows with $K$. This accounts for the increasing separation between the identification curves. We also find a larger gain when half the streams change, since more streams can contribute to lowering the threshold.

The different baseline intensities help explain why identification is incomplete at time 200. A change in a low-intensity stream generates fewer additional counts, and hence less likelihood evidence, than the same multiplicative change in a high-intensity stream. Many low-intensity changes therefore remain undetected at the reporting time. The improvement from e-d-BH concerns the fraction identified by that time, rather than recovery of every changed stream. We also find fewer false discoveries with the e-d-Bonferroni procedures, which control their expected number; these comparisons do not match the achieved FDR.

\subsubsection{Error under no changes}

We also run each model with no changes. In Table~\ref{tab:diverse-null}, we report FDR at time 200, which under the global null is the probability of making any declaration at that time. All four procedures have small estimated error probabilities in these settings. Together with the alternative results, we find that the BH threshold improves identification while keeping the observed FDR below the nominal bound. These are evaluations at specified reporting times; they do not estimate the supremum over stopping times in EOP. In \Cref{app:internal-comparison}, we also compare our procedures with separate SR charts that do not adjust for multiple streams.

\begin{table}[H]
\centering
\caption{Global-null FDR $\downarrow$ at time 200. Under the global null, this is the probability of any current declaration. Entries are estimates (Monte Carlo standard errors) from 1500 runs. The conditional mean of each Poisson count depends on its previous count. The SR threshold for e-d-BH is $8000K/r$; the e-d-Bonferroni thresholds for SR and CUSUM are $8000K$. SHEREF uses SR threshold $4000K/r$ at time 200, where $r$ is the step-up rank. Zero estimates and standard errors mean that no false declarations were observed at this time.}
\label{tab:diverse-null}
\begingroup\fontsize{9}{11}\selectfont
\setlength{\tabcolsep}{3pt}\renewcommand{\arraystretch}{1.12}
\begin{tabular}{@{}llrrrr@{}}
\hline
 & & \multicolumn{4}{c}{FDR $\downarrow$} \\
Setting & $K$ & \shortstack{e-d-BH\\(ours)} & \shortstack{e-d-Bonferroni\\SR (ours)} & \shortstack{e-d-Bonferroni\\CUSUM (ours)} & SHEREF \\
\hline
Heavy tails, $\rho=0$ & 100 & $0.0007\;(0.0007)$ & $0.0007\;(0.0007)$ & $0.0000\;(0.0000)$ & $0.0007\;(0.0007)$ \\
Heavy tails, $\rho=0.5$ & 100 & $0.0007\;(0.0007)$ & $0.0007\;(0.0007)$ & $0.0000\;(0.0000)$ & $0.0013\;(0.0009)$ \\
Poisson counts & 20 & $0.0007\;(0.0007)$ & $0.0007\;(0.0007)$ & $0.0000\;(0.0000)$ & $0.0040\;(0.0016)$ \\
Poisson counts & 100 & $0.0007\;(0.0007)$ & $0.0007\;(0.0007)$ & $0.0000\;(0.0000)$ & $0.0007\;(0.0007)$ \\
Poisson counts & 500 & $0.0000\;(0.0000)$ & $0.0000\;(0.0000)$ & $0.0000\;(0.0000)$ & $0.0007\;(0.0007)$ \\
\hline
\end{tabular}
\endgroup
\end{table}

\FloatBarrier
\subsection{Global detection}\label{sec:global-detection}

Identifying streams is more demanding than declaring that some change has occurred. To compare the latter task, we add Sum-CUSUM \citep{mei2010efficient}, the scan and mixture procedure $T_2$ of \citet{xie2013sequential}, and the transformed-CUSUM procedure of \citet{chan2017optimal}. We fit thresholds to global-null ARL 1000 and validate them on independent runs. We use 50 streams, shifts of size 0.4 or 0.75, and either 5 or 25 changes at time 101.

We find that the global procedures usually raise an alarm sooner, particularly when many streams change. They can combine evidence across the changed streams without deciding which ones to report. We use this comparison to complement the identification lower bound in \Cref{thm:gaussian-optimality}. A common ARL also allows different probabilities of an alarm before the change. We report those probabilities, validated ARLs, conditional delays and the methods' tuning choices in \Cref{app:global-comparison}.
\par\endgroup

\section{Discussion}

We propose EOP as a Type I error metric for sequential monitoring, and construct e-d-BH and e-d-Bonferroni to control it. The metric allows the user to choose when to report while relating the permitted error to expected monitoring time. We prove guarantees under the global-filtration condition \eqref{eq:validity}, allowing dependence within and across streams, and show that they imply lower bounds on the expected time to false declarations. Thus the procedures provide run-length guarantees as well as bounds on errors in the set reported at a stopping time. We also show how normalizing the e-detectors gives FDR control at each fixed time, and how using e-processes gives FDR control at stopping times without the patience denominator.

For Gaussian changes, we prove finite-sample bounds on the observations sufficient to identify a prescribed number of changes. These bounds quantify the effect of the shift size and the number of changed streams on detection. For independent streams, we also prove a matching first-order delay lower bound under the FDR EOP constraint. Both procedures attain the minimax leading term as the error level tends to zero with the number of streams fixed. The lower bound requires control over every changepoint configuration: evidence that other streams have changed cannot by itself justify declaring a change in a stream that may still be null.

We find that the EOP-controlling procedures also have useful identification power in the experiments. e-d-BH reports more true changes and reports sooner than the time-adjusted Gandy--Lau procedure in the fixed-change settings. The comparisons also show the benefit of information supplied to competing methods: Bayesian procedures can report sooner when the data follow their changepoint prior, and global scan and aggregation methods can combine evidence without identifying the affected streams. Between our two procedures, we find that the BH threshold improves identification over Bonferroni, including with heavy-tailed observations and dependent counts. We report ordinary errors alongside these gains, since controlling different error criteria does not imply the same achieved FDR.

Our optimality result fixes the number of streams. Studying a growing number of streams may distinguish the BH and Bonferroni thresholds beyond the leading term proved here. Another question is how much detection delay is needed when the post-change distribution must be estimated. We also allow one changepoint per stream and evaluate the current set of declarations. To extend the guarantees to permanent removal of declared streams or repeated changes, we would need to account for errors accumulated over successive decisions.

\section*{Acknowledgements}
The vast majority of the theoretical results in this paper were obtained without the use of AI. We used AI tools to write code for several of the experiments and to help restructure the paper. The authors take full responsibility for the correctness, originality and intellectual integrity of the paper, including all theoretical results, experimental code and reported findings.

\printbibliography
\clearpage
\appendix
\crefalias{section}{appendix}
\section{Proofs for error over patience}\label{app:proofs-eop}

\subsection[Proof of the impossibility result]{Proof of \Cref{prop:impossibility}}\label{app:proof-impossibility}

\begin{proof}
Finite expectation implies that $\tau_\eta^*(\varphi)$ is finite almost surely. At that time there are at least $\eta$ declarations, all of which are false under the global null.
\end{proof}

\subsection[Proof of the partial-null run-length bounds]{Proof of \Cref{prop:partial}}\label{app:proof-partial}

\begin{proof}
We need only consider finite expected stopping times. At $\tau_\eta^*(\varphi)$, at most $q$ of the $N_\tau\geq\eta$ declarations can be true. Hence $V_\tau\geq\eta-q\geq1$ and $V_\tau/N_\tau\geq1-q/\eta$. Apply the corresponding EOP bound.
\end{proof}

\section{Proofs for multiple testing with e-detectors}\label{app:proofs-methods}

\subsection{Validity of the SR and CUSUM constructions}\label{sec:proofconstruction}

\begin{proof}
We prove that \eqref{eq:factor} implies \eqref{eq:validity} for both recursions in \eqref{eq:recursions}. Fix a stream and write $c=\xi^{(k)}-1$, omitting the stream index. Before the change, both recursions satisfy
\begin{align*}
    \E[M_t\mid\F_{t-1}]\leq M_{t-1}+1.
\end{align*}
Finite-time integrability of the process stopped at $c$ follows by induction. Therefore
$M_{t\wedge c}-(t\wedge c)$ is a supermartingale. Optional stopping at the bounded time $\tau\wedge n$ gives
\begin{align*}
    \E[M_{\tau\wedge n\wedge c}]\leq\E[\tau\wedge n\wedge c].
\end{align*}
Fatou's lemma on the left and monotone convergence on the right give \eqref{eq:validity}. This also covers $c=0$, since $M_0=0$. The product representations \eqref{eq:restart} follow by induction from the recursions.
\end{proof}

\subsection[Proof of the EOP bounds]{Proof of \Cref{thm:eop}}\label{app:proof-eop}

\begin{proof}
If e-d-BH declares stream $k$ at $\tau$, then $M_\tau^{(k)}\geq K/(\alpha N_\tau)$. Thus, pathwise,
\begin{align}
    \frac{V_\tau}{N_\tau\vee1}
    &\leq\frac{\alpha}{K}\sum_{k=1}^K
        M_\tau^{(k)}\indicator_{\{\tau<\xi^{(k)}\}}
    \leq\frac{\alpha}{K}\sum_{k=1}^K
        M_{\tau\wedge(\xi^{(k)}-1)}^{(k)}.\label{eq:selfconsistency}
\end{align}
Taking expectations and using \eqref{eq:validity} proves \eqref{eq:bhbound}. For e-d-Bonferroni, use
$\indicator_{\{k\in R_\tau^{\mathrm{Bonf}}\}}\leq\beta M_\tau^{(k)}/K$,
sum over the streams still pre-change, and apply the same argument.
\end{proof}

\subsection[Proof of the false-alarm run-length bounds]{Proof of \Cref{cor:arl}}\label{app:proof-arl}

\begin{proof}
When the relevant expectations are finite, apply \Cref{thm:eop} at the corresponding first-alarm times. Under the global null, the FDP at the first declaration is one. At the first $\eta$ false declarations, $V_\tau\geq\eta$. At the first false declaration under any configuration, the FDP is at least $1/K$; if at most $q$ streams change, it is at least $1/(q+1)$.
\end{proof}

\subsection[Proof of fixed-time control]{Proof of \Cref{prop:fixed}}\label{app:proof-normalized}

\begin{proof}
At a fixed time $t$, write $Z_t^{(k)}=M_t^{(k)}/t$. Each still-null stream satisfies $\E[Z_t^{(k)}]\leq1$. The e-BH inequality gives
\begin{align*}
    \fdr(\varphi,\xi,t)
    \leq\frac{\alpha}{K}\sum_{k\in\mathcal H_0(t)}\E[Z_t^{(k)}]
    \leq\alpha\frac{|\mathcal H_0(t)|}{K}.
\end{align*}
The same calculation bounds PFER for e-Bonferroni, using $\indicator_{\{Z_t^{(k)}\geq K/\beta\}}\leq\beta Z_t^{(k)}/K$ and summing over $k\in\mathcal H_0(t)$.

For the run-length claim, assume \eqref{eq:validity}.
Let $T$ be the first rejection time of e-d-BH applied to $M_t^{(k)}/t$ at level $\alpha<1$. If $\E_\infty[T]$ were finite, then $T<\infty$ almost surely and
\begin{align*}
    \sum_{k=1}^K M_T^{(k)}\geq KT/\alpha.
\end{align*}
Taking expectations and applying \eqref{eq:validity} under the global null gives
\begin{align*}
    \frac{K}{\alpha}\E_\infty[T]
    \leq\sum_{k=1}^K\E_\infty[M_T^{(k)}]
    \leq K\E_\infty[T],
\end{align*}
which is impossible since $1\leq\E_\infty[T]<\infty$ and $\alpha<1$.
\end{proof}

\subsection[Proof of e-process control]{Proof of \Cref{prop:eprocesses}}\label{app:proof-eprocesses}

\begin{proof}
Replacing the e-detectors by e-processes in \eqref{eq:selfconsistency} and applying \eqref{eq:eprocessvalid} gives FDR control at level $\alpha$. The same substitution in the Bonferroni argument gives PFER control at level $\beta$.

Under the global null, let $T$ be the first rejection time for e-BH at level $\alpha$. On $\{T\leq n\}$, we have $K^{-1}\sum_k E_T^{(k)}\geq1/\alpha$. Nonnegativity and \eqref{eq:eprocessvalid} at $T\wedge n$ therefore give
\begin{align*}
    \P_\infty(T\leq n)
    \leq\frac{\alpha}{K}\sum_{k=1}^K\E_\infty[E_{T\wedge n}^{(k)}]
    \leq\alpha.
\end{align*}
Letting $n\to\infty$ yields $\pfa\leq\alpha$. When $\alpha<1$, we have $\P_\infty(T=\infty)\geq1-\alpha>0$, so the ARL is infinite. For e-Bonferroni, a declaration requires at least one $E_T^{(k)}\geq K/\beta$, and the same argument with $\beta$ in place of $\alpha$ gives its PFA bound and infinite ARL when $\beta<1$.

For the weighted restart construction \eqref{eq:weighted-eprocess}, write $c=\xi^{(k)}-1$. Under \eqref{eq:factor}, the process $E_{t\wedge c}^{(k)}$ is a nonnegative supermartingale with initial value $\sum_jw_j\leq1$. Indeed, the weight $w_t$ moves from the unused tail to a new likelihood product at time $t$, and every active product has conditional expected multiplier at most one. Optional stopping at $\tau\wedge n$ and Fatou's lemma give \eqref{eq:eprocessvalid}.
\end{proof}

\section{Proofs for Gaussian changes}\label{app:proofs-power}

Throughout this appendix, we write $I=\mu^2/2$ and $b_r=\log\{K/(\alpha r)\}$.

\subsection[Proof of the power bound]{Proof of \Cref{thm:power}}\label{app:proof-power}

\begin{proof}
For a changed stream $k\in S$, both detectors dominate the likelihood product starting at $\nu$:
\begin{align*}
    M_{\nu+n-1}^{(k)}\geq\exp(Z_{k,n}),
    \qquad Z_{k,n}=\sum_{t=\nu}^{\nu+n-1}\log L_t^{(k)}.
\end{align*}
Conditional on $\F_{\nu-1}$, each $Z_{k,n}$ has distribution $\mathcal N(nI,2nI)$. For $nI\geq b_r$, the Gaussian tail bound gives
\begin{align*}
    \P(Z_{k,n}<b_r\mid\F_{\nu-1})
    \leq\exp\left\{-\frac{(nI-b_r)^2}{4nI}\right\}.
\end{align*}
Write $a=\log\{m/[\delta(m-r+1)]\}$. Condition \eqref{eq:power-rate} gives $nI\geq2b_r+4a$. Since $(2b_r+4a-b_r)^2-4a(2b_r+4a)=b_r^2\geq0$, and $(x-b_r)^2/(4x)$ increases for $x\geq b_r$, the probability above is at most $e^{-a}$.

Let $Y$ count the changed streams with $Z_{k,n}\geq b_r$. On $\{Y\geq r\}$, e-d-BH declares all these $Y$ streams. Indeed, any such stream among the largest $r$ detector values is selected; a stream of rank $j>r$ meets the smaller threshold $K/(\alpha j)$ and forces at least $j$ selections. Thus $T_n\geq Y\indicator_{\{Y\geq r\}}$. Markov's inequality now gives
\begin{align*}
    \P(T_n<r\mid\F_{\nu-1})
    &\leq\P(m-Y\geq m-r+1\mid\F_{\nu-1})\\
    &\leq\frac{m e^{-a}}{m-r+1}=\delta.
\end{align*}
We used only the marginal Gaussian distributions and independence over time, so dependence between streams does not affect this argument.
\end{proof}

\subsection{Sharper bounds for independent streams}\label{app:further-power}

We can improve the power bound when the streams are independent. Let $\Phi$ denote the standard normal distribution function and define
\begin{align}
    p_{r,n}=\Phi\left(\frac{nI-b_r}{\sqrt{2nI}}\right),
    \qquad \pi_{r,n}=\P\{\Bin(m,p_{r,n})\geq r\},\label{eq:powerconstants}
\end{align}
where $\Bin(m,p)$ denotes a binomial random variable with $m$ trials and success probability $p$. Here $p_{r,n}$ is the probability that one changed stream's likelihood product exceeds $K/(\alpha r)$.

\begin{proposition}\label{prop:further-power}
Under the assumptions of \Cref{thm:power}, if the streams are also independent, then
\begin{align}
    \P(T_n\geq r\mid\F_{\nu-1})&\geq\pi_{r,n},\label{eq:fractionpower}\\
    \E[T_n/m\mid\F_{\nu-1}]
        &\geq\max_{1\leq r\leq m}
        p_{r,n}\P\{\Bin(m-1,p_{r,n})\geq r-1\}.\label{eq:tdrpower}
\end{align}
\end{proposition}

The second bound controls the true discovery rate, the expected proportion of changed streams currently declared. Its complement is the expected proportion missed, which differs from the usual false non-discovery rate: the latter divides by the number of non-declarations.

\begin{proof}
The count $Y$ in the proof of \Cref{thm:power} now has distribution $\Bin(m,p_{r,n})$, giving \eqref{eq:fractionpower}. For \eqref{eq:tdrpower}, use $T_n\geq Y\indicator_{\{Y\geq r\}}$ and
\begin{align*}
    \E[Y\indicator_{\{Y\geq r\}}]
    =mp_{r,n}\P\{\Bin(m-1,p_{r,n})\geq r-1\}.
\end{align*}
Dividing by $m$ and maximizing over $r$ proves the claim.
\end{proof}

\subsection{Detection delay}\label{app:proof-delay}

For the model in \Cref{thm:power}, including dependent streams, let
\begin{align*}
    n_r=\left\lceil\frac{2b_r+4\log\{2m/(m-r+1)\}}{I}\right\rceil.
\end{align*}
We prove that
\begin{align}
    \E[D_r\mid\F_{\nu-1}]\leq2n_r.\label{eq:delay-rate}
\end{align}
Thus the expected delay has the same order as the number of observations in \eqref{eq:power-rate} for fixed $\delta$.

\begin{proof}
Divide the post-change observations into consecutive groups of $n_r$ observations. At the end of each group, both detectors dominate the likelihood product starting at its beginning. The proof of \Cref{thm:power}, with $\delta=1/2$, shows that the observations in that group alone suffice to declare $r$ changed streams with probability at least $1/2$. The groups are independent, so the probability that the first $\ell$ all fail is at most $2^{-\ell}$. Summing these probabilities bounds the expected number of groups by two, giving \eqref{eq:delay-rate}.
\end{proof}

When the streams are independent, the same reasoning gives the sharper bounds
\begin{align}
    \P(D_r>\ell n\mid\F_{\nu-1})
        &\leq(1-\pi_{r,n})^\ell,\qquad \ell=0,1,\ldots,\label{eq:delaytail}\\
    \E[D_r\mid\F_{\nu-1}]
        &\leq\inf_{n\geq1}\frac{n}{\pi_{r,n}}.\label{eq:delaymean}
\end{align}
Here we use the sufficient event with probability $\pi_{r,n}$ from \Cref{prop:further-power} for each group of $n$ observations.

For fixed $K,m,\mu$ and $\alpha\downarrow0$, we can also improve the leading constant in \eqref{eq:delay-rate}:
\begin{align}
    \E[D_r\mid\F_{\nu-1}]
    \leq\frac{\log\{K/(\alpha r)\}}{I}\{1+o(1)\}.\label{eq:firstorder}
\end{align}
This holds even for dependent streams.
\begin{proof}
Take $n=\lceil(b_r+b_r^{3/4})/I\rceil$. Then $p_{r,n}\to1$ and $n/(b_r/I)\to1$. By the same Markov inequality used in \Cref{app:proof-power}, each group succeeds with probability at least
$1-m(1-p_{r,n})/(m-r+1)$, which tends to one. The expected delay is at most $n$ divided by this probability, proving \eqref{eq:firstorder}.
\end{proof}

For completeness, a sharper sufficient sample size for declaring all changed streams with probability at least $1-\varepsilon$ is
\begin{align}
    n_\varepsilon=
    \left\lceil\frac{\{\sqrt{\log(m/\varepsilon)}
                    +\sqrt{b_m+\log(m/\varepsilon)}\}^2}{I}\right\rceil.
    \label{eq:explicitdelay}
\end{align}
The Gaussian tail bound in \Cref{app:proof-power} makes each stream's failure probability at most $\varepsilon/m$ at this sample size. A union bound gives
$\P(T_{n_\varepsilon}=m\mid\F_{\nu-1})\geq1-\varepsilon$, and hence
$\P(D_m\leq n_\varepsilon\mid\F_{\nu-1})\geq1-\varepsilon$.

For e-d-Bonferroni, replace each $b_r$ by $\log(K/\beta)$. Each stream crossing this threshold is selected, so the proofs above still hold. The asymptotic statement then takes $\beta\downarrow0$ with $K,m,\mu$ fixed.

\subsection[Proof of Gaussian optimality]{Proof of \Cref{thm:gaussian-optimality}}\label{app:proof-optimality}

\begin{proof}
First fix $S$, $\nu$, and an integer $n\geq1$, and put $h=\nu+n-1$.
For $k\in S$, let
\begin{align*}
 A_k&=\{\varphi_t^{(k)}=1\text{ for some }\nu\leq t\leq h\},\\
 \tau_k&=\inf\{t\geq\nu:\varphi_t^{(k)}=1\}\wedge h,
 \qquad Q_k=\P_{S\setminus\{k\},\nu}.
\end{align*}
The stopping time $\tau_k$ is bounded by $h$. Under $Q_k$, stream $k$
never changes and the FDP at $\tau_k$ is at least
$\indicator_{A_k}/K$. The FDR EOP bound therefore gives
\begin{align}
 Q_k(A_k)\leq K\E_{Q_k}[\fdp_{\tau_k}]
 \leq K\alpha\E_{Q_k}[\tau_k]\leq K\alpha h.
 \label{eq:gaussian-early-declaration}
\end{align}
The two laws differ only in stream $k$ from time $\nu$ onwards. Their
log likelihood ratio on $\F_h$ is
\begin{align*}
 Z_{k,n}=\sum_{t=\nu}^{h}(\mu X_t^{(k)}-\mu^2/2)
 \sim\mathcal N(nI,2nI)\quad\text{under }\P_{S,\nu}.
\end{align*}
For $\gamma>nI$, change of measure and the Gaussian tail bound yield
\begin{align}
 \P_{S,\nu}(A_k)
 &\leq e^\gamma Q_k(A_k)+\P_{S,\nu}(Z_{k,n}>\gamma)\notag\\
 &\leq K\alpha h e^\gamma+
       \exp\left\{-\frac{(\gamma-nI)^2}{4nI}\right\}.
 \label{eq:gaussian-change-measure}
\end{align}
The event $\{D_r\leq n\}$ requires $A_k$ for at least $r$ members of
$S$. Consequently,
\begin{align}
 \P_{S,\nu}(D_r\leq n)
 \leq\frac mr\left[
     K\alpha h e^\gamma+
     \exp\left\{-\frac{(\gamma-nI)^2}{4nI}\right\}\right].
 \label{eq:gaussian-finite-lower}
\end{align}
This bound holds for every algorithm in $\mathcal A_\alpha$.

For $\varepsilon\in(0,1)$, write $L=\log(1/\alpha)$ and take
$n=\lfloor(1-\varepsilon)L/I\rfloor$ and
$\gamma=(1-\varepsilon/2)L$. For small enough $\alpha$, $n\geq1$,
$\gamma-nI\geq\varepsilon L/2$, and \eqref{eq:gaussian-finite-lower}
is at most
\begin{align*}
 \frac mr\left[
   K(\nu+n-1)\alpha^{\varepsilon/2}
   +\alpha^{\varepsilon^2/[16(1-\varepsilon)]}\right].
\end{align*}
It tends to zero for fixed $\nu$. Thus
$\E_{S,\nu}[D_r]\geq n\P_{S,\nu}(D_r>n)$ gives $\liminf_{\alpha\downarrow0}I\E_{S,\nu}[D_r]/\log(1/\alpha)\geq1$ after letting
$\varepsilon\downarrow0$. At $\nu=1$, the same lower bound is uniform
over $\mathcal A_\alpha$, and each $\mathcal D_{m,r}(\varphi)$ is at
least $\E_{S,1}[D_r]$. This proves the required minimax lower bound.

For the upper bound, e-d-BH belongs to $\mathcal A_\alpha$ by \Cref{thm:eop}. The same is true of e-d-Bonferroni at $\beta=\alpha$, since FDR is bounded by PFER. For these procedures, \eqref{eq:firstorder} and its Bonferroni counterpart give
\begin{align*}
    \E_{S,\nu}[D_r\mid\F_{\nu-1}]
    \leq \frac{b}{I}\{1+o(1)\},
\end{align*}
where $b=\log\{K/(\alpha r)\}$ for BH and $b=\log(K/\alpha)$ for Bonferroni. The deterministic bound in that proof depends only on $K,m,r,\mu$ and $\alpha$, so we can take the suprema in \eqref{eq:worst-delay}. Since $b/\log(1/\alpha)\to1$, both procedures attain \eqref{eq:gaussian-minimax}.
\end{proof}

\section{Further results}\label{app:further-results}

\subsection{Dependence and the global filtration}\label{app:filtration}

We give examples that explain the role of the global filtration in \eqref{eq:validity}. Suppose $X_t^{(1)}$ are i.i.d.\ standard Gaussian and $X_t^{(2)}=X_{t+1}^{(1)}$. Observing the second stream reveals the next observation in the first. A stopping rule using both streams therefore has more information than one using only the first stream, and a likelihood-ratio process valid for the latter need not be valid for the former.

The following two constructions allow dependence while satisfying \eqref{eq:validity}. First, suppose
\begin{align*}
    (X_t^{(1)},\ldots,X_t^{(K)})\mid\F_{t-1}\sim\mathcal N(\mu_t,\Sigma_t),
\end{align*}
where $\mu_t$ and $\Sigma_t$ may depend on the past. Before stream $k$ changes, assume $\mu_t^{(k)}\leq0$ and $(\Sigma_t)_{kk}=\sigma_k^2$ is known. The off-diagonal entries need not be known. Then, for any nonnegative $\F_{t-1}$-measurable $\lambda_t^{(k)}$,
\begin{align}
    L_t^{(k)}
    =\exp\left\{\lambda_t^{(k)}X_t^{(k)}
                 -(\lambda_t^{(k)})^2\sigma_k^2/2\right\}\label{eq:gaussianfactor}
\end{align}
satisfies \eqref{eq:factor} before stream $k$ changes.

Second, suppose the pre-change conditional distribution of $X_t^{(k)}$ is symmetric about zero given $\F_{t-1}$. Choose a function $h_t^{(k)}:\R\to[-1,1]$ using the global past, and require it to be odd: $h_t^{(k)}(-x)=-h_t^{(k)}(x)$. Then
\begin{align}
    L_t^{(k)}=1+h_t^{(k)}(X_t^{(k)})\label{eq:symmetryfactor}
\end{align}
has conditional mean one. For example, take $h_t^{(k)}(x)=\lambda_t^{(k)}\operatorname{sign}(x)$, where $\lambda_t^{(k)}\in[-1,1]$ is $\F_{t-1}$-measurable and $\operatorname{sign}(0)=0$. Applying \eqref{eq:recursions} gives a nonparametric detector without specifying a density or requiring moments. Conditional symmetry of each coordinate is sufficient, even when the streams are dependent.

\subsection{Other discounts}\label{sec:weighted}

In \eqref{eq:discount}, we changed the denominator while keeping the procedure fixed. We can instead choose the detector weights to obtain a bound with the desired denominator. Suppose $g:[0,\infty)\to[0,\infty)$ is nondecreasing and concave, with $g(0)=0$ and $g(1)=1$. Let $w_t=g(t)-g(t-1)$, and replace the SR recursion by
\begin{align}
    M_0^{(k),g}=0,\qquad
    M_t^{(k),g}=(M_{t-1}^{(k),g}+w_t)L_t^{(k)}.\label{eq:weighted}
\end{align}
Write $c=\xi^{(k)}-1$. Under \eqref{eq:factor}, the process
$M_{t\wedge c}^{(k),g}-g(t\wedge c)$ is a supermartingale, so the optional-stopping argument in \Cref{sec:proofconstruction} gives
\begin{align*}
    \E[M_{\tau\wedge(\xi^{(k)}-1)}^{(k),g}]
    \leq\E[g(\tau\wedge(\xi^{(k)}-1))].
\end{align*}
Let $\varphi^g$ denote e-d-BH applied to these statistics. Using \eqref{eq:selfconsistency} and Jensen's inequality yields
\begin{align}
    \fdr(\varphi^g,\xi,\tau)
    &\leq\frac{\alpha}{K}\sum_k
             \E[g(\tau\wedge(\xi^{(k)}-1))]
      \leq\alpha g(\E[\tau]).\label{eq:weightedbound}
\end{align}
The analogous PFER bound holds for e-d-Bonferroni. The weights therefore give error control with denominator $g(\E[\tau])$ directly.

To see how these weights affect detection, consider a change at time $\nu$ and write $b_r=\log\{K/(\alpha r)\}$. The likelihood product starting at $\nu$ now has weight $w_\nu$. When $w_\nu>0$, its sufficient log threshold becomes $b_r+\log(1/w_\nu)$. For $g(t)=t^p$, $0<p<1$, the additional term is $(1-p)\log\nu-\log p+o(1)$ as $\nu\to\infty$. Later changes therefore require more evidence from this product. The linear choice $g(t)=t$ assigns weight one to every start. This is also what allows the delay proof in \Cref{app:proof-delay} to use the same detection probability for successive groups of observations. When the weights decrease, the sufficient evidence threshold changes from one group to the next.

\subsection{Global null testing}

If we only want to declare that some stream has changed, we can use the e-detector global null testing (e-d-GNT) rule
\begin{align*}
    \varphi_t^*=\indicator_{\{\sum_{k=1}^K M_t^{(k)}\geq K/\alpha\}}.
\end{align*}
Define its global error at $\tau$ as
\begin{align*}
    \ger(\varphi^*,\xi,\tau)
       =\E\left[\varphi_\tau^*
                  \indicator_{\{\tau<\min_k\xi^{(k)}\}}\right].
\end{align*}
By nonnegativity and \eqref{eq:validity},
\begin{align*}
    \ger(\varphi^*,\xi,\tau)
    &\leq\frac{\alpha}{K}\sum_k
        \E[M_\tau^{(k)}\indicator_{\{\tau<\xi^{(k)}\}}]
    \leq\frac{\alpha}{K}\sum_k\E[\tau\wedge(\xi^{(k)}-1)]
    \leq\alpha\E[\tau].
\end{align*}
This gives a bound on error over patience and a global-null ARL of at least $1/\alpha$. The rule raises an alarm whenever e-d-Bonferroni at the same level makes a declaration. It can also combine evidence from several streams to raise an alarm before any individual threshold is crossed, but does not identify the changed streams. Here we sum the nonnegative detector values; the procedure of \citet{mei2010efficient} instead sums local log-CUSUM statistics.

\subsection{Holm's procedure at a fixed time}

At a deterministic time $t$, $P_t^{(k)}=\min\{1,t/M_t^{(k)}\}$ is a valid p-value for each $k\in\mathcal H_0(t)$, with $t/0=\infty$. Indeed, Markov's inequality gives
$\P(P_t^{(k)}\leq u)\leq u$ for $u\in(0,1)$.
Applying Holm's procedure \citep{holm1979simple} to these p-values therefore controls FWER at level $\alpha$. In detector notation, it declares the largest $r_t$ values, where
\begin{align*}
    r_t=\max\left\{r\in[K]:
       M_t^{[i]}\geq\frac{t(K-i+1)}{\alpha}
       \text{ for every }i\leq r\right\},
\end{align*}
and $r_t=0$ when the set is empty. This rule includes the fixed-time Bonferroni declarations. Its FWER guarantee holds at each deterministic time, as does the validity of the normalized e-values used to construct the p-values.

\subsection{Choosing a level after observing the data}

The inequality \eqref{eq:selfconsistency} holds for each realization of the data and every $\alpha\in(0,1)$, using the same detector values. Write $\varphi(\alpha)$ for the e-d-BH rule at level $\alpha$, and let $\fdp(\varphi(\alpha),\xi,\tau)$ denote its false discovery proportion at $\tau$. For any fixed $\tau\in\mathcal T_\P$,
\begin{align*}
    \E\left[\sup_{\alpha\in(0,1)}
       \frac{\fdp(\varphi(\alpha),\xi,\tau)}{\alpha}\right]
    \leq\frac1K\sum_k\E[\tau\wedge(\xi^{(k)}-1)].
\end{align*}
The same argument gives the analogous inequality for e-d-Bonferroni with $V_\tau/\beta$. In particular, a level chosen after observing the data can be substituted for $\alpha$ inside the expectation. The guarantee averages over both the data and the selected level; it does not condition on that level.

\begingroup
\raggedbottom
\section{Experimental details and further results}\label{app:simulation-details}

\subsection{External identification methods}\label{app:external-details}

In \Cref{sec:external-identification}, we use $K=100$, $\mu=0.5$, $H=400$ and $q=0.05$, with 3000 paired runs for each setting. All observations are independent conditional on the changepoints. We use the Gaussian likelihood factors $L_t^{(k)}=\exp\{\mu X_t^{(k)}-\mu^2/2\}$. For fixed-change settings, the first 10 or 50 streams change at time 101 or 251 and the rest never change. In the prior-generated setting, changepoints are independent with
\begin{align*}
    \P(\xi^{(k)}=\infty)=\tfrac12,
    \qquad \P(\xi^{(k)}=t)=\tfrac12(0.01)(0.99)^{t-1},\quad t\geq1.
\end{align*}
We use this same prior for the two Bayesian methods in every setting, including those with fixed changepoints.

For Gandy--Lau, we implement Algorithm 2 of \citet{gandy2013nonrestarting}. Starting from $C_0^{(k)}=0$, set
\begin{align*}
    C_t^{(k)}=\phi\bigl(\min\{20,\max\{0,C_{t-1}^{(k)}+\log L_t^{(k)}\}\}\bigr),
\end{align*}
where $\phi$ rounds to the nearest point of $\{0,0.025,\ldots,20\}$. We compute the null distribution by the finite-state Markov recursion, integrating the Gaussian increment density over each rounding interval. Let $\overline F_t(c)=\P_\infty(C_t\geq c)$ be the tail probability of a null chart following this recursion. We take $p_t^{(k)}=\overline F_t(C_t^{(k)})$, the inclusive null tail for the observed statistic. We fixed the upper boundary before the simulations, large enough to permit the smallest BH cutoff $q/(KH)$. We apply ordinary BH at level $q$ for the original fixed-time comparison. We also use level $q/H$, labeled time-adjusted throughout. For any $\tau\leq H$, this adjustment gives
\begin{align*}
    \E[\fdp_\tau]
    =\sum_{t=1}^H\E[\fdp_t\indicator_{\{\tau=t\}}]
    \leq\sum_{t=1}^H\E[\fdp_t]\leq q.
\end{align*}
This uses the fixed-time FDR guarantee under independence, and does not assume independence between reporting decisions at different times.

For the Bayesian methods, let $\pi_t=\P(\xi^{(k)}=t)$ and $s_t=\P(\xi^{(k)}>t)$ under the stated prior. The posterior odds and posterior null probability are
\begin{align*}
    G_t^{(k)}=\frac{s_{t-1}G_{t-1}^{(k)}+\pi_t}{s_t}L_t^{(k)},
    \qquad z_t^{(k)}=\frac1{1+G_t^{(k)}},\qquad G_0^{(k)}=0.
\end{align*}
Before the first declaration, MD-FDR \citep[Algorithm 1]{chen2020false} declares the largest $r$ posterior odds, where
\begin{align*}
    r=\max\{j\in[K]:G_t^{[j]}\geq K/(qj)-1\},
\end{align*}
with $r=0$ for an empty set. The posterior LFDR procedure \citep[Algorithms 3--4]{lu2022optimal} instead orders the null probabilities as $z_t^{(1)}\leq\cdots\leq z_t^{(K)}$ and chooses
\begin{align*}
    r=\max\left\{j\in[K]:\frac1j\sum_{i=1}^j z_t^{(i)}\leq q\right\}.
\end{align*}
It declares these $r$ streams. This is their LFDR specialization with the utility that rewards detections. Both methods subsequently remove declared streams, but our comparison ends at their first nonempty batch. For MD-FDR, every selected posterior null probability in that batch is at most $qr/K\leq q$; for posterior LFDR, their average is at most $q$. Averaging over the data and the stated prior therefore gives FDR at most $q$ for the first report from either method. This argument does not give the same bound for every fixed configuration.

We evaluate all first-report errors and true counts at the same decision: the first nonempty declared set by time 400, with an empty report if no declaration occurs. A declaration before a stream changes is false. Mean true counts include zero for pre-change reports and for no report. We also give the delay to a post-change report conditional on no earlier report, with pre-change report probabilities shown alongside it. These conditioning sets differ across methods. A report need not contain a true discovery; its FDR and true count are therefore needed to interpret its delay. We exclude reports after the horizon and assign nondetections the remaining monitoring duration. We observed no nondetections in the alternative runs. For the random-prior setting, we measure delay from the earliest changepoint in each run.

In Figure~\ref{fig:external-fixedtime}, we continue monitoring after the first report and examine the current declaration sets. This lets us compare how quickly the frequentist procedures identify the remaining changed streams.

\begin{figure}[H]
\centering
\includegraphics[width=\textwidth]{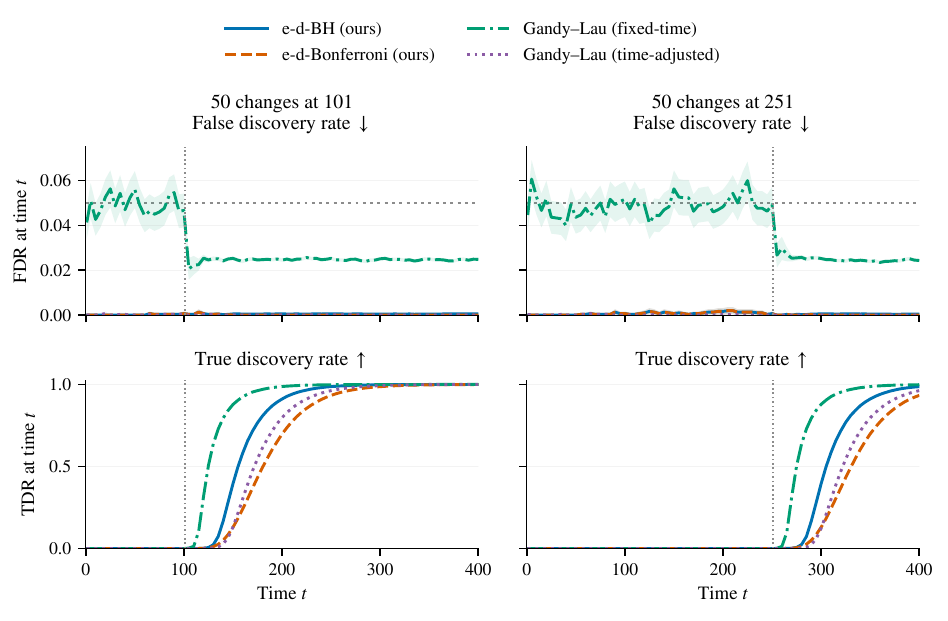}
\caption{Fixed-time FDR and proportion of changed streams currently declared, with 50 of 100 Gaussian streams changing at time 101 or 251. Gandy--Lau uses level $q=0.05$; its time-adjusted version uses $q/H$. With $H=400$ and $\alpha=q/H$, our thresholds are $800{,}000/r$ for e-d-BH at rank $r$, and $800{,}000$ for e-d-Bonferroni. Shading shows 1.96 Monte Carlo standard errors over 3000 runs. These are current declaration sets from continued monitoring, distinct from the first reports in Table~\ref{tab:external-main}.}
\label{fig:external-fixedtime}
\end{figure}

We find that Gandy--Lau at level $q$ identifies changes earliest, while its estimated fixed-time FDR remains near or below $q$. In the separate global-null runs, we estimate its FDR at times 100, 200 and 400 as $0.051\;(0.004)$, $0.054\;(0.004)$ and $0.047\;(0.004)$, respectively, where parentheses give Monte Carlo standard errors. These estimates agree with its fixed-time guarantee, even though its first-report FDR under the global null is $0.906\;(0.005)$. The two observations are compatible: a small chance of a false declaration at each specified time can yield a large chance of at least one false declaration over the monitoring period. Stopping at that first declaration therefore changes the error criterion substantially.

Moving the change from time 101 to 251 has little effect on the curves when measured from the change. After 50 post-change observations, we find that e-d-BH identifies about $40\%$ of the changed streams, compared with about $13\%$ for either e-d-Bonferroni or time-adjusted Gandy--Lau, in both settings. The additional time before the change nevertheless gives more opportunities for an early false report. For unadjusted Gandy--Lau, its probability rises from $0.605$ to $0.807$. Thus the similar post-change identification curves need not lead to similar first reports. Our continued-monitoring comparison measures how well the current set identifies changes as evidence accumulates; the first-report comparison also measures the consequences of choosing when to report.

\begin{table}[H]
\centering
\caption{First reports when 10 of 100 streams change. FDR, true count and the probability of a pre-change report are unconditional. Delay counts observations from the change to a report, conditional on no pre-change report; it need not be a true discovery. All runs reported by time 400, so the no-report probability and its empirical MC SE are zero for every row. Entries are estimates (Monte Carlo standard errors) from 3000 runs. Conditional sample sizes and uncapped-delay details are retained in the accompanying CSV. With $K=100$ and $\alpha=0.05/400$, our SR procedures use $M_t^{[r]}\geq800{,}000/r$ for e-d-BH (rank $r$) and $M_t^{(k)}\geq800{,}000$ for e-d-Bonferroni.}
\label{tab:external-sparse}
\begingroup\fontsize{9}{11}\selectfont
\setlength{\tabcolsep}{3pt}\renewcommand{\arraystretch}{1.12}
\begin{tabular}{@{}lrrrr@{}}
\hline
Method & FDR $\downarrow$ & True count $\uparrow$ & Delay $\downarrow$ & Pre-change $\downarrow$ \\
\hline
\multicolumn{5}{c}{Change at 101} \\
e-d-BH (ours) & .0053 (.0013) & 1.376 (.013) & 40.80 (.19) & .0020 (.0008) \\
e-d-Bonferroni (ours) & .0050 (.0013) & 1.031 (.004) & 41.24 (.19) & .0020 (.0008) \\
Gandy--Lau & .6455 (.0086) & 0.599 (.018) & 19.49 (.22) & .6083 (.0089) \\
Gandy--Lau (time adj.) & .0027 (.0009) & 1.314 (.012) & 51.42 (.22) & .0027 (.0009) \\
MD-FDR & .0110 (.0019) & 1.368 (.014) & 39.13 (.19) & .0047 (.0012) \\
Posterior LFDR & .7451 (.0079) & 0.297 (.010) & 15.30 (.23) & .6903 (.0084) \\
\hline
\multicolumn{5}{c}{Change at 251} \\
e-d-BH (ours) & .0163 (.0023) & 1.375 (.014) & 41.17 (.18) & .0133 (.0021) \\
e-d-Bonferroni (ours) & .0160 (.0023) & 1.020 (.004) & 41.60 (.19) & .0133 (.0021) \\
Gandy--Lau & .8250 (.0069) & 0.287 (.014) & 19.57 (.29) & .8063 (.0072) \\
Gandy--Lau (time adj.) & .0057 (.0014) & 1.301 (.012) & 52.44 (.21) & .0053 (.0013) \\
MD-FDR & .0229 (.0027) & 1.310 (.013) & 46.66 (.20) & .0203 (.0026) \\
Posterior LFDR & .9192 (.0050) & 0.090 (.006) & 22.04 (.48) & .9147 (.0051) \\
\hline
\end{tabular}
\endgroup
\end{table}

\begin{table}[H]
\centering
\caption{Reporting delays and pre-change reports for the 50-change settings in Table~\ref{tab:external-main}. Delay counts observations from the change to a report, conditional on no pre-change report; it need not be a true discovery. Pre-change-report probabilities are unconditional, so the methods condition on different sets of runs. No reports were censored at 400 in these settings. Entries are estimates (Monte Carlo standard errors) from 3000 runs; conditioning counts are in the accompanying CSV. With $K=100$ and $\alpha=0.05/400$, our SR procedures use $M_t^{[r]}\geq800{,}000/r$ for e-d-BH (rank $r$) and $M_t^{(k)}\geq800{,}000$ for e-d-Bonferroni.}
\label{tab:external-delays}
\begingroup\fontsize{9}{11}\selectfont
\setlength{\tabcolsep}{3pt}\renewcommand{\arraystretch}{1.12}
\begin{tabular}{@{}lrrrr@{}}
\hline
 & \multicolumn{2}{c}{Change at 101} & \multicolumn{2}{c}{Change at 251} \\
Method & Delay $\downarrow$ & Pre-change $\downarrow$ & Delay $\downarrow$ & Pre-change $\downarrow$ \\
\hline
e-d-BH (ours) & 29.31 (.11) & .0020 (.0008) & 29.10 (.12) & .0200 (.0026) \\
e-d-Bonferroni (ours) & 30.02 (.12) & .0020 (.0008) & 29.81 (.13) & .0200 (.0026) \\
Gandy--Lau & 11.81 (.11) & .6053 (.0089) & 11.95 (.16) & .8073 (.0072) \\
Gandy--Lau (time adj.) & 37.89 (.13) & .0040 (.0012) & 38.71 (.13) & .0047 (.0012) \\
MD-FDR & 27.48 (.12) & .0063 (.0014) & 33.45 (.13) & .0243 (.0028) \\
Posterior LFDR & 9.16 (.13) & .6700 (.0086) & 14.46 (.33) & .9153 (.0051) \\
\hline
\end{tabular}
\endgroup
\end{table}

\begin{table}[H]
\centering
\caption{First reports under the correctly specified changepoint prior. Each stream independently has no change with probability $1/2$; otherwise its first post-change observation is geometric with hazard 0.01. The Bayesian FDR guarantees average over this prior. Delay is from the first actual change in the family to a report, conditional on no report before that change. It is not a delay to a true discovery. Every run reported by time 400. Entries are estimates (Monte Carlo standard errors), with 3000 runs. With $K=100$ and $\alpha=0.05/400$, our SR procedures use $M_t^{[r]}\geq800{,}000/r$ for e-d-BH (rank $r$) and $M_t^{(k)}\geq800{,}000$ for e-d-Bonferroni. A superscript $a$ means zero events observed, with empirical MC SE zero; it does not assert a zero error probability.}
\label{tab:external-prior}
\begingroup\fontsize{9}{11}\selectfont
\setlength{\tabcolsep}{3pt}\renewcommand{\arraystretch}{1.12}
\begin{tabular}{@{}lrrrr@{}}
\hline
Method & FDR $\downarrow$ & True count $\uparrow$ & Delay $\downarrow$ & Pre-change $\downarrow$ \\
\hline
e-d-BH (ours) & $0^{a}$ & 1.467 (.016) & 52.84 (.17) & $0^{a}$ \\
e-d-Bonferroni (ours) & $0^{a}$ & 1.054 (.004) & 53.27 (.18) & $0^{a}$ \\
Gandy--Lau & .3469 (.0084) & 0.874 (.015) & 15.89 (.18) & .0547 (.0042) \\
Gandy--Lau (time adj.) & .0030 (.0010) & 1.214 (.011) & 48.03 (.26) & .0003 (.0003) \\
MD-FDR & .0007 (.0005) & 1.494 (.016) & 48.56 (.16) & $0^{a}$ \\
Posterior LFDR & .0379 (.0033) & 1.117 (.008) & 26.12 (.11) & $0^{a}$ \\
\hline
\end{tabular}
\endgroup
\end{table}

\FloatBarrier
\subsection{Comparing e-d-BH and e-d-Bonferroni}\label{app:internal-comparison}

We give further results for the Gaussian experiment in \Cref{sec:identification-delay}, using $K=100$ independent streams, $H=400$ and changes of size $0.5$ at time 101. We use 4000 paired runs per case and $\alpha=0.05/400$, with the same SR statistics for both procedures. This lets us examine the effect of the multiple testing rule while keeping the evidence fixed. We also include separate SR charts at threshold $1/\alpha$, which adjust for monitoring time but not the number of streams.

Figure~\ref{fig:tradeoff} compares the probability of any declaration under no changes with identification when 50 streams change. We estimate the global-null alarm probability as 0.031 for each of our procedures, compared with 0.970 for the separate charts. The identification curves show the cost of the larger Bonferroni threshold: it takes longer for changed streams to reach it. BH uses the evidence in other streams to lower the threshold and identifies a larger fraction sooner. Its guarantee controls the false proportion, while Bonferroni controls the number of false declarations; both use the same level $\alpha$ here.

\begin{figure}[!htbp]
\centering
\includegraphics[width=\textwidth]{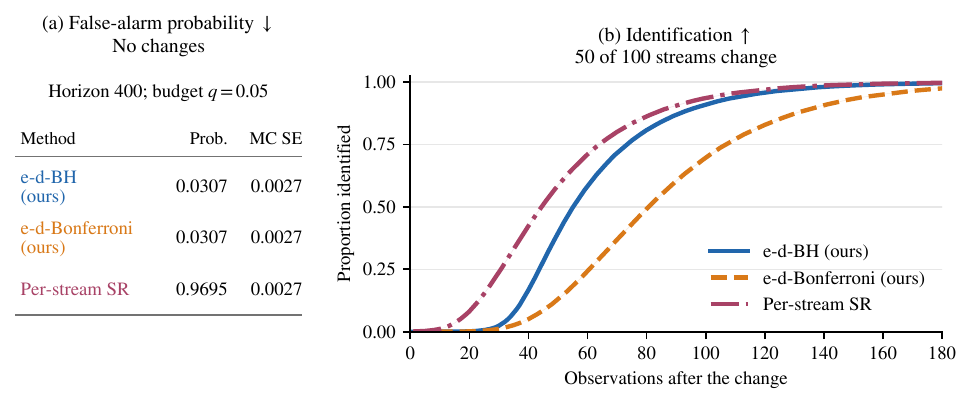}
\caption{Comparison within our proposed family and with separate SR charts. Left: probability of any declaration by time 400 under the global null, with Monte Carlo standard errors in parentheses; $q=0.05$. Right: proportion of changed streams currently declared when 50 of 100 streams change at time 101. With $\alpha=q/400$, our thresholds are $800{,}000/r$ for e-d-BH at rank $r$ and $800{,}000$ for e-d-Bonferroni. The separate SR threshold is $8000$. Curve shading shows 1.96 Monte Carlo standard errors over 4000 runs.}
\label{fig:tradeoff}
\end{figure}

The separate charts identify changes quickly, but their guarantee applies to each stream individually. To see this, let $T$ be a null stream's first crossing of $1/\alpha$, capped at $H$. By \eqref{eq:validity}, its crossing probability is at most $\alpha\E[M_T^{(k)}]\leq\alpha H=q$. Summing across streams gives only $Kq$, which is uninformative here. Thus the high collection-wide alarm probability in Figure~\ref{fig:tradeoff} is consistent with the individual guarantees.

Figure~\ref{fig:tradeoff-all} compares 10 and 50 changes and adds SHEREF. For a fixed collection of streams, SHEREF applies e-BH to the SR statistics divided by the current time, at fixed-time FDR level $q=0.05$. We find a larger gap between e-d-BH and e-d-Bonferroni when more streams change. More streams can then contribute to increasing the selected rank $r$, which lowers the BH threshold $K/(\alpha r)$. The Bonferroni threshold remains $K/\alpha$ in both settings.

\begin{figure}[!htbp]
\centering
\includegraphics[width=\textwidth]{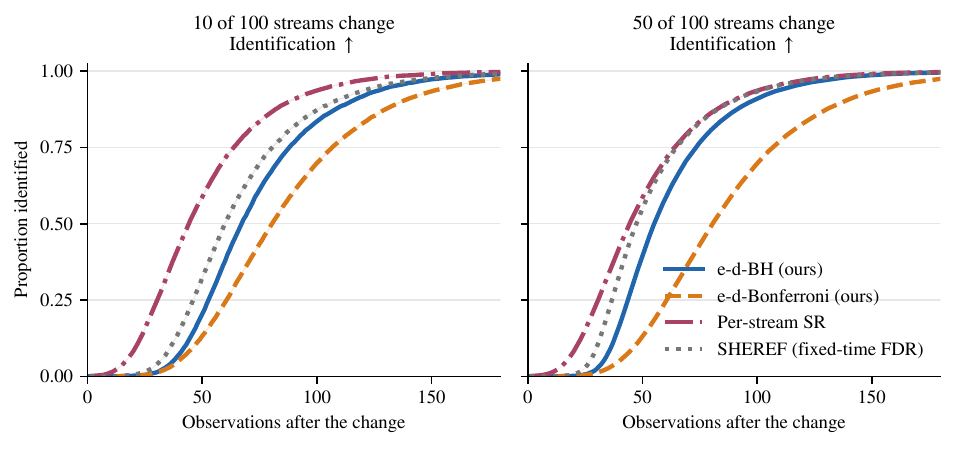}
\caption{Expected proportion of changed streams currently declared, for 10 or 50 changes among 100 streams. Our thresholds are $800{,}000/r$ for e-d-BH at rank $r$ and $800{,}000$ for e-d-Bonferroni, with $\alpha=0.05/400$. The separate SR threshold is $8000$; SHEREF uses threshold $2000t/r$ for fixed-time FDR level 0.05. Shading shows 1.96 Monte Carlo standard errors over 4000 runs.}
\label{fig:tradeoff-all}
\end{figure}

These curves explain the identification-time comparison in Table~\ref{tab:tradeoff}. Requiring several changed streams to be declared together describes a different task from making the first report in \Cref{sec:external-identification}.

SHEREF has a smaller threshold than e-d-BH before $H$: its threshold $Kt/(qr)$ is $t/H$ times our threshold $KH/(qr)$. This explains its earlier identification in Figure~\ref{fig:tradeoff-all}; the declaration sets agree at $H$. Its error guarantee concerns fixed times. Under the global null, we estimate its FDR as 0.0015 at time 400 and 0.059 at the first declaration capped at 400. We therefore assess the faster identification together with the reporting rule and its error criterion.

We count the first post-change observation as one when measuring delay. If a run has not identified 80\% of the changes by time 400, we assign time 401 and retain it in the mean. No such nondetections occurred for this target in these simulations. The monitoring rules continue after an earlier declaration; the identification times depend on the true changed set and serve only to assess performance.

We specified both alternatives and all reporting rules before running the experiment. The accompanying data files also retain results at the first one or ten declarations, capped at the horizon, and at time 400.

\FloatBarrier
\subsection{Global detection at a common ARL target}\label{app:global-comparison}

We use $K=50$ independent Gaussian streams with shifts $\mu\in\{0.4,0.75\}$ in either 5 or 25 streams at time 101. We fit all thresholds to global-null ARL 1000 and compare our e-d-BH and e-d-Bonferroni procedures with Sum-CUSUM \citep{mei2010efficient}, the scan and mixture procedure $T_2$ of \citet{xie2013sequential}, and the transformed-CUSUM procedure of \citet{chan2017optimal}.

For our CUSUM inputs and Sum-CUSUM, write
\begin{align*}
    W_0^{(k)}=0,\qquad
    W_t^{(k)}=\max\{0,W_{t-1}^{(k)}+\mu X_t^{(k)}-\mu^2/2\}.
\end{align*}
e-d-Bonferroni thresholds $\max_k W_t^{(k)}$, while Sum-CUSUM thresholds $\sum_k W_t^{(k)}$. For our SR inputs, we threshold $\max_r\{\log M_t^{[r]}+\log r\}$ for e-d-BH and $\max_k\log M_t^{(k)}$ for e-d-Bonferroni. A threshold $h$ on these log statistics corresponds to e-detector thresholds $e^h/r$ for e-d-BH and $e^h$ for e-d-Bonferroni. We report the fitted $h$ values in Table~\ref{tab:global-external}. For Xie--Siegmund, write $S_{k,t}=\sum_{s=1}^tX_s^{(k)}$ and $U_{k,t,\ell}=(S_{k,t}-S_{k,t-\ell})/\sqrt\ell$. Their $T_2$ procedure thresholds
\begin{align*}
    \max_{1\leq\ell\leq\min(t,199)}
    \sum_{k=1}^K\log\left[1-p_0+p_0\exp\{(U_{k,t,\ell}\vee0)^2/2\}\right],
    \qquad p_0=0.1.
\end{align*}
The window restriction and mixing weight follow their numerical setup. This procedure estimates the positive shifts and does not use their true size. Chan's procedure instead runs local CUSUM statistics $R_t^{(k)}$ with design shift $\mu_0=1$ and thresholds
\begin{align*}
    \sum_{k=1}^K\log\left[1+p_0\{\lambda_M\exp(R_t^{(k)}/2)-1\}\right],
    \qquad p_0=0.1,\quad\lambda_M=0.64.
\end{align*}
These parameters also follow the published numerical setup and stay fixed across alternatives. Our likelihood ratios and those of Sum-CUSUM use the true shift. The three external procedures make global declarations, so we do not assign them stream-identification sets.

We fit thresholds for our procedures and Sum-CUSUM using 800 null runs, then validate their ARLs with 1500 fresh runs continued until every alarm occurs. Xie--Siegmund and Chan use 400 independent null training runs and 750 fresh, uncensored validation runs. There is no adjustment using the validation data. Every alternative uses the same 2000 observation paths for all methods. We report the actual validated ARLs rather than assuming exact equality.

\begin{table}[tb]
\centering
\caption{Global detection with thresholds fitted to null ARL 1000. The fitted statistic threshold $h$ is shown to four decimal places; the statistic definitions are given in the text. For our procedures, the corresponding evidence thresholds are $e^h/r$ for BH (at rank $r$) and $e^h$ for Bonferroni. Other entries give independently validated ARL or conditional detection delay, with Monte Carlo standard errors in parentheses. Delay conditions on no alarm before time 101. Xie--Siegmund and Chan use 750 validation runs; the other methods use 1500. Every alternative uses the same 2000 runs.}
\label{tab:global-external}
\begingroup\fontsize{9}{11}\selectfont
\setlength{\tabcolsep}{4pt}\renewcommand{\arraystretch}{1.12}
\begin{tabular}{@{}lrrrr@{}}
\toprule
Method & $h$ & Validated ARL & \multicolumn{2}{c}{Conditional delay $\downarrow$} \\
 & & & 5 changes & 25 changes \\
\midrule
\multicolumn{5}{l}{Shift $\mu=0.4$} \\
\addlinespace[2pt]
e-d-BH (ours) & 10.5188 & 1025 (25) & 43.55 (0.34) & 26.76 (0.18) \\
e-d-Bonf./SR (ours) & 10.5156 & 1029 (25) & 44.10 (0.35) & 27.93 (0.20) \\
e-d-Bonf./CUSUM (ours) & 7.7936 & 1028 (25) & 43.96 (0.36) & 27.46 (0.20) \\
Sum-CUSUM & 58.0955 & 992 (25) & 37.21 (0.39) & 8.01 (0.07) \\
Xie--Siegmund T2 & 12.2527 & 955 (35) & 33.85 (0.38) & 9.08 (0.09) \\
Chan transformed CUSUM & 2.0659 & 1009 (36) & 43.65 (0.76) & 8.79 (0.10) \\
\midrule
\multicolumn{5}{l}{Shift $\mu=0.75$} \\
\addlinespace[2pt]
e-d-BH (ours) & 10.3810 & 1023 (26) & 15.64 (0.11) & 10.22 (0.06) \\
e-d-Bonf./SR (ours) & 10.3810 & 1023 (26) & 15.79 (0.11) & 10.55 (0.07) \\
e-d-Bonf./CUSUM (ours) & 8.6791 & 1058 (28) & 15.66 (0.11) & 10.42 (0.07) \\
Sum-CUSUM & 52.8639 & 991 (25) & 13.26 (0.13) & 2.99 (0.02) \\
Xie--Siegmund T2 & 12.2527 & 955 (35) & 11.14 (0.10) & 3.36 (0.03) \\
Chan transformed CUSUM & 2.0659 & 1009 (36) & 11.13 (0.11) & 3.62 (0.03) \\
\bottomrule
\end{tabular}
\endgroup
\end{table}

We measure global detection delay conditional on no alarm before the change, assigning delay 301 to nondetections by time 400. We observed no nondetections for Xie--Siegmund or Chan. Threshold fitting does not retain the nominal EOP level used in the identification experiments. Here we compare global detection at a common ARL target, rather than procedures at a common FDR EOP level.

We find that the scan and aggregation procedures benefit from combining evidence about whether any change has occurred, especially when 25 streams change. We also find differences in their pre-change behavior despite the similar ARLs. For 25 changes of size 0.4, we estimate the probability of an alarm before time 101 as 0.026 for our two SR procedures, 0.036 for our CUSUM procedure, 0.033 for Sum-CUSUM, and approximately 0.092 for both Xie--Siegmund and Chan. Monte Carlo standard errors range from 0.0036 to 0.0065. We retain these probabilities and all conditional sample sizes in the accompanying results.

\FloatBarrier
\subsection{Heavy-tailed observations and dependent counts}
\label{app:diverse-details}

We use 1500 independent runs for each of 15 configurations, with a monitoring horizon of $H=400$. We set $q=0.05$ and $\alpha=q/H=0.000125$. Given factors $L_t^{(k)}$, the SR and restart-CUSUM statistics are
\[
M_t^{(k)}=(1+M_{t-1}^{(k)})L_t^{(k)},\qquad
C_t^{(k)}=\max\{1,C_{t-1}^{(k)}\}L_t^{(k)},
\quad M_0^{(k)}=C_0^{(k)}=0.
\]
The latter is the maximum product over nonempty restart windows and can be smaller than one. e-d-BH uses the SR inputs and thresholds $K/(r\alpha)=8000K/r$. We apply e-d-Bonferroni to both SR and CUSUM statistics at threshold $K/\alpha=8000K$. The resulting guarantees concern FDR for e-d-BH and PFER for e-d-Bonferroni, each at level $q$ for any stopping time bounded by $H$. We also run SHEREF with inputs $M_t^{(k)}/t$ and fixed-time FDR level $q$. Its threshold $Kt/(rq)$ is smaller than e-d-BH's $KH/(rq)$ before $H$, and the declaration sets agree at $H$. Its earlier identification therefore does not establish the same guarantee at a data-dependent stopping time.

\subsubsection{Heavy-tailed observations}

Let all $Z_t^{(k)}$ and $U_t$ be independent Student-$t_3$ variables divided by $\sqrt{3}$, and define
\[
X_t^{(k)}=\mu_t^{(k)}+\sqrt{1-\rho}\,Z_t^{(k)}+\sqrt{\rho}\,U_t,
\qquad K=100,\quad\rho\in\{0,0.5\}.
\]
Vectors are independent over time. The six configurations are the global null, simultaneous changes in streams 1--20 at time 101, and staggered changes in streams 1--10 at time 101 and streams 11--20 at time 201, each at both values of $\rho$. Changed streams have shifts alternating between 0.45 and 0.9, so each group of ten contains five of each size. The remaining streams have mean zero throughout. The factors are
\[
L_t^{(k)}=1+0.5\operatorname{sign}(X_t^{(k)}).
\]
For every still-null stream, conditional symmetry gives $\E[L_t^{(k)}\mid\mathcal F_{t-1}]=1$. This construction uses symmetry, rather than the Student-$t$ likelihood or an exponential-moment assumption. We report the simultaneous changes in Figure~\ref{fig:diverse-heavy-full}. In Figure~\ref{fig:diverse-staggered-full}, we examine the staggered changes, keeping the same noise distribution and shift sizes.

\begin{figure}[H]
\centering
\includegraphics[width=\textwidth]{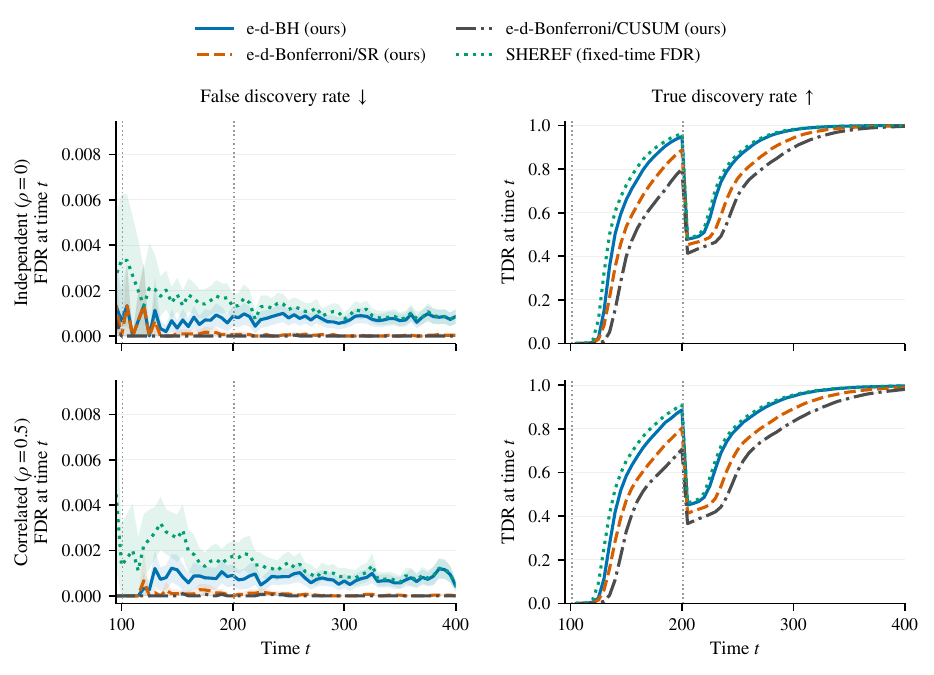}
\caption{FDR $\downarrow$ and TDR $\uparrow$ with ten heterogeneous shifts at time 101 and another ten at time 201, among 100 heavy-tailed streams. TDR counts only already changed streams, so its denominator increases at the second change. Our thresholds are $800{,}000/r$ for e-d-BH at rank $r$, and $800{,}000$ for both e-d-Bonferroni procedures, with $\alpha=0.05/400$. SHEREF uses threshold $2000t/r$. Shading shows 1.96 Monte Carlo standard errors over 1500 runs.}
\label{fig:diverse-staggered-full}
\end{figure}

We find that e-d-BH retains its identification advantage after the second group changes, for both independent and correlated streams. Declarations from the first group can lower the threshold for the newly changed streams. For example, with ten streams declared, the e-d-BH threshold is $80{,}000$, compared with $800{,}000$ for e-d-Bonferroni. The later group can therefore benefit from evidence already accumulated in other streams, without knowing when either group changes.

TDR drops at time 201 because its denominator increases from ten to twenty, adding streams that have only just changed. We use the fraction of already changed streams so that the curves describe identification at the displayed time. We find that all four procedures eventually identify nearly all twenty changes in this experiment; e-d-BH reaches high identification rates sooner than either e-d-Bonferroni procedure. The FDR curves remain small through both changes. As in the simultaneous-change experiment, SHEREF uses a lower threshold before time 400 and has a fixed-time guarantee.

\subsubsection{Dependent Poisson counts}

For the count experiment, $K\in\{20,100,500\}$ and the known intercepts alternate between $a_k=0.1$ and $a_k=1$. Streams are independent across $k$. Conditional on the observed past, let
\[
\lambda_t^{(k)}=a_k+0.3X_{t-1}^{(k)},\qquad
X_t^{(k)}\sim\operatorname{Poisson}(c_t^{(k)}\lambda_t^{(k)}).
\]
The multiplier $c_t^{(k)}$ is one before a change and 1.75 afterwards. Since $0.3(1.75)<1$, both regimes have finite stationary means. Starting with $X=0$, we generate 200 null warmup observations and retain the last count as observed initial history; only the detector statistics are reset. At each $K$, we run the global null and configurations in which the first 10\% or 50\% of streams change at time 101. Every affected subset has equal numbers from both intercept groups. These give nine configurations. The exact conditional likelihood factors are
\[
L_t^{(k)}=\exp\{X_t^{(k)}\log(1.75)-0.75\lambda_t^{(k)}\}.
\]
They have conditional mean one for every still-null stream, despite the serial dependence. The intensity model and all its parameters are known. In Figure~\ref{fig:diverse-counts-full}, we report results at the prespecified time 200 for both affected fractions.

We compute Monte Carlo standard errors across independent runs, retaining within-run dependence. Shading and bars show 1.96 standard errors, clipped to $[0,1]$ for rates. Zero estimates mean no observed errors. All summaries concern current declaration sets.
\par\endgroup

\end{document}